\documentclass{amsart}

\makeatletter
\def\subsection{\@startsection{subsection}{2}%
  \z@{.5\linespacing\@plus.7\linespacing}
{.5\baselineskip}%
  {\bfseries\centering}%
}
\makeatother

\usepackage[hyphens]{url}

\RequirePackage{calc}
\RequirePackage[dvipsnames,table]{xcolor}
\RequirePackage{hyperref} %
\RequirePackage{lastpage}

\RequirePackage{amsmath,amsthm,amssymb} %
\numberwithin{equation}{section} %
\RequirePackage{graphicx} %
\RequirePackage{caption} 
\usepackage{appendix}
\RequirePackage{enumitem}
\RequirePackage[sc]{mathpazo} %
\RequirePackage[T1]{fontenc} %
\usepackage[export]{adjustbox}

\RequirePackage[capitalize,nameinlink,noabbrev]{cleveref} %

\usepackage{Commands}

\newtheorem{thm}{Theorem}[section]
\newtheorem{theorem}[thm]{Theorem}
\newtheorem{conjecture}[thm]{Conjecture}
\newtheorem{mainprob}{Problem}

\newtheorem{prob}{Problem}

\newtheorem{lemma}[thm]{Lemma}

\newtheorem{corollary}[thm]{Corollary}
\newtheorem*{cor*}{Corollary}

\newtheorem{example}[thm]{Example}

\newtheorem{proposition}[thm]{Proposition}
\newenvironment{axioms}[2][] 
 {
 \enumerate[label=(#2\arabic*#1),
 ref=\textup{(#2\arabic*#1)}]}
 {\endenumerate}

\theoremstyle{definition}

\newtheorem{definition}[thm]{Definition}
\newtheorem*{fact}{Fact}
\newtheorem{note}[thm]{Note}

\usepackage{subcaption}
\usepackage{Commands}

\title{Towards plethystic \texorpdfstring{$\sl_2$}{sl(2)} crystals}

\author[\'{A}. Guti\'{e}rrez]{
\'{A}lvaro Guti\'{e}rrez
}
\thanks{The author was funded by the University of Bristol Research Training Support Grant.}

\date{\today}

\address{School of Mathematics, University of Bristol, UK}
\email{\href{mailto:a.gutierrezcaceres@bristol.ac.uk}{a.gutierrezcaceres@bristol.ac.uk}.}

\begin{document}

\maketitle


\begin{abstract}
To find crystals of $\mathfrak{sl}_2$ representations of the form $\Lambda^n\textup{Sym}^r\mathbb{C}^2$ it suffices to solve the combinatorial problem of decomposing the Young lattice into symmetric, saturated chains. We review the literature on this latter problem, and present a strategy to solve it. For $n \le 4$, the strategy recovers recently discovered solutions. We obtain (i) counting formulas for plethystic coefficients, (ii) new recursive formulas for plethysms of Schur functions, and (iii) formulas for the number of constituents of $\Lambda^n\textup{Sym}^r\mathbb{C}^2$.
\end{abstract}



\section{Introduction}

Consider the Lie algebra $\sl_2 = \sl_2(\CC)$ with the natural action on $\CC^2$. 
Two classic facts are (i) the finite dimensional irreducible representations of $\sl_2$ are given by the symmetric powers $\Sym^r\CC^2$ for $r \in \Z_{\ge0}$, and (ii) if $V$ is a representation of $\sl_2$ then so is the alternating power $\Alt^nV$ for all $n \in \Z_{\ge0}$.
\begin{mainprob}\label{prob: plethysm}
    Decompose $\Alt^n \Sym^r \CC^2$ into irreducible representations of $\sl_2$.
\end{mainprob}
This is one of the easiest cases of the problem of \emph{plethysm}, and notoriously difficult to tackle \cite{StanleyList, mystery}. Using classical notation, the problem asks to find the multiplicities $a_{1^n[r]}^k$ fitting in
\begin{equation}
    \label{eq: plethystic coefficients}
\Alt^n \Sym^r \CC^2 = \bigoplus_{k} (\Sym^k\CC^2)^{\oplus a_{1^n[r]}^k}.
\end{equation}
As it is often in algebraic combinatorics, we ask for a solution that is \emph{explicit and positive}, expressing $a_{1^n[r]}^k$ as the cardinality of a set given by quasipolynomial equations and inequalities ---this solves Problem \ref{prob: plethysm} in the sense of \cite{StanleyList, Pak}.

A solution to the deceptively similar problem of decomposing $\Sym^a \CC^2 \otimes \Sym^b \CC^2$ into irreducible representations of $\sl_2$ goes back to Clebsch and Gordan in the \textsc{xix} century \cite{Gordan}. The tensor product problem is nowadays best understood through crystal theory, and Kashiwara's tensor product rule \cite{Kashiwara}.

For our purposes, an $\sl_2$ crystal is a directed graph attached to a representation of $\sl_2$ and whose vertices are weighted by elements of $\frac{1}{2}\Z$, see Figure \ref{subfig: crystal sl2}. We follow \cite{BS} for the main concepts in crystals, but delay the details to \S\ref{sec: crystals and characters}.
If there is an arc $x\longrightarrow y$, then the weight of $y$ satisfies $\wt(y) = \wt(x) - 1$.
Decomposing a representation into irreducibles translates to decomposing its crystal into connected components.
For $\sl_2$, each finite dimensional irreducible representation $\Sym^k\CC^2$ has a crystal which is a {path graph} in $k+1$ vertices and {weight-symmetric}. For instance, the crystal in Figure \ref{subfig: crystal sl2} decomposes as the crystal of $\Sym^2\CC^2\oplus\Sym^6\CC^2$.

\begin{figure}
    \begin{subfigure}{0.6\textwidth}
    \centering
    \includegraphics[]{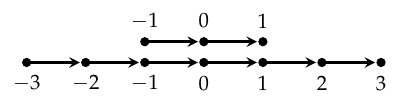}
    \\ \vspace{1cm}
    \subcaption{The crystal of $\Alt^2\Sym^4\CC^2$.}
    \label{subfig: crystal sl2}
    \end{subfigure}
    \begin{subfigure}{0.38\textwidth} \centering \includegraphics{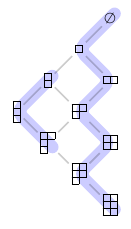}
    \subcaption{A decomposition of $L(2,3)$.}
    \label{subfig: L(2,3)}
    \end{subfigure}
    \caption{Computing the crystal of $\Alt^2\Sym^4\CC^2$ shows $\Alt^2\Sym^4\CC^2 \cong \Sym^2\CC^2 \oplus \Sym^6\CC^2$, solving Problem \ref{prob: plethysm}. The crystal can be computed by solving Problem  \ref{prob: SCD} for $L(2,3)$.}
    \label{fig: crystal sl2 and L(2,3)}
\end{figure}

Vertices of a crystal form a basis of the corresponding representation.
A canonical basis of
$\Alt^n\Sym^r \CC^2$ is in bijection with the set $L(n,m)$ of partitions whose Young diagram fits into an $n \times m$ rectangle, where $r+1 = m+n$. One such bijection is given in \eqref{eq: bijection B(n) L(n)}.
Endow $L(n,m)$ 
with the partial order given by the
Young’s lattice (Figure \ref{subfig: L(2,3)}); that is, with the covering relation $\lambda \lessdot \mu$ if the Young diagram of $\lambda$ is obtained from that of $\mu$ by removing one box.

With these ingredients, a natural way of tackling Problem \ref{prob: plethysm} is to solve the following:
\begin{mainprob}\label{prob: SCD}
    Decompose the poset $L(n,m)$ into rank-symmetric, saturated chains.
\end{mainprob}
In Figure \ref{fig: crystal sl2 and L(2,3)}, we illustrate how to go from a solution to Problem~\ref{prob: SCD} to a solution of Problem~\ref{prob: plethysm}.
We reiterate that, for these purposes,
we require an explicit solution to Problem \ref{prob: SCD}, in which the set of `highest weight elements' of the decomposition (see \S\ref{sec: crystals and characters}) is fully described as a set given by some quasipolynomial equations and inequalities \cite{OSSZ}.\medskip

We begin by reviewing the literature for Problem \ref{prob: SCD}. 
Our main contribution is to develop a crystal theoretic framework to tackle Problem \ref{prob: SCD}.
For $0 \le n \le 4$, the framework produces explicit expressions for the crystal operators which are reminiscent of Kashiwara's tensor product rule. In the language of posets, these are explicit matchings that produce symmetric chain decompositions of $L(n,m)$. We fully describe the highest weight elements, solving Problem \ref{prob: plethysm} for $n\le 4$. 
We remark Problem \ref{prob: plethysm} was already solved for $n\le 4$ (see \cite{mystery} for a short survey), and a solution for $n\ge5$ was announced to us by Pak and Panova in private communication.

Our constructions for Problem \ref{prob: SCD} have properties in common with several works in the literature (see the end of \S \ref{sec: literature}) ---this suggests that we may be near to a canonical solution that will 
solve the cases for higher $n$, as stated in Conjecture \ref{conjecture}.

We retrieve counting formulas for the coefficients $\smash{a_{1^n[r]}^{k}}$ involved in \eqref{eq: plethystic coefficients}. After Propositions~\ref{p: recover 3} and \ref{p: recover 4}, these recover similar formulas found in \cite{OSSZ} through a bijection.

Taking characters on both sides of \eqref{eq: plethystic coefficients} we get
\[
{s_{(1^n)}\circ s_{(r)} (q^{-\frac{1}{2}},q^{\frac{1}{2}})} =
s_{(1^n)} (q^{-\frac{r}{2}}, q^{1-\frac{r}{2}}, q^{2-\frac{r}{2}}, \ldots, q^{\frac{r}{2}}) = 
\sum_{k} a_{1^n[r]}^k \cdot s_{(k)}(q^{-\frac{1}{2}},q^{\frac{1}{2}}),
\]
where $s_\lambda$ are Schur functions and $\circ$ is the plethysm of symmetric functions.
We obtain new recursive formulas for the plethysm of Schur functions
$\smash{s_{(1^n)}\circ s_{(r)} (q^{-\frac{1}{2}},q^{\frac{1}{2}})}$ for $n = 3$ and $4$. This plethysm is the (centred) $q$-binomial $\smash{\qbinom{r+1}{n}}$, which we prefer for a cleaner notation; see \S\ref{sec: crystals and characters} for a definition.
With this notation, the displayed equation immediately above becomes
\[
\qbinom{r+1}{n} = 
\sum_{k} a_{1^n[r]}^k \cdot [k+1],
\]
where $[k+1]$ is a $q$-integer.
For example, Figure \ref{fig: crystal sl2 and L(2,3)} is lifting $\qbinom{~5~}{2} = [3] + [7]$. Our formulas express $\smash{\qbinom{r+1}{n}}$ in terms of $\smash{\qbinom{~r~}{n}}$ and $q$-binomials of the form $\smash{\qbinom{*}{n-2}}$. To state them, we define an operator on 
$q$-characters.
\begin{definition}[Plus operator]
Given $f = \sum_{i} d_i\cdot[i]$ we let $f_{+j} = \sum_{i} d_i\cdot[i+j]$.
\end{definition}

\begin{theorem}\label{thm: char 3}
    The character of $\Alt^3\Sym^r\CC^2$ satisfies the recursion
    \begin{equation*}\label{eq: recursion 3}
        \qbinom{r+1}{3} = \qbinom{r}{3}_{+3} +\sum ~[r-4k-1],
    \end{equation*}
    where the sum ranges over all $k\ge0$ such that $4k<r-1 - 2\delta_{r~\text{odd}}$.
\end{theorem}
\begin{theorem}\label{thm: char 4}
    The character of $\Alt^4\Sym^r\CC^2$ satisfies the following recursion:
\[
\qbinom{r+1}{4} = \qbinom{r}{4}_{+4} + \sum_{k\ge0} \qbinom{r-6k-1-3\delta_{r~\text{even}}}{2} + \sum_{k\ge0} \qbinom{r-6k-4-3\delta_{r~\text{odd}}}{2}_{+6}.
\]
\end{theorem}
These formulas are non-trivial even as counting formulas for $q=1$, and we do not know any other way of deriving them besides with the crystals we define below. These are in particular different than the ones found in \cite{OSSZ, Treteault}.


The number of constituents of $\Alt^n\Sym^r \CC^2$ is $\#\{\lambda\in L(n,m)  :  \lambda \vdash \lfloor{mn/2}\rfloor\}$.
We obtain in Corollary \ref{cor: number 2} 
that this number is
$\lfloor{(r+1)/2}\rfloor$ for $n\!=\!2$ and $\lfloor{(r+1)^2/8}\rfloor$ for $n\!=\!3$. It is roughly $\lfloor{(r+1)^3/36}\rfloor$ for $n\!=\!4$. These results appear already in \cite[p.~69]{AF}; we give two new and simpler proofs: one is fully combinatorial, the other algebraic. 

The paper is organized as follows: a literature review is given in \S\ref{sec: literature}, after which we introduce some basic definitions in \S\ref{sec: preliminaries}. We present our framework in \S\ref{sec: crystals n} and obtain the explicit constructions for $n \le 4$ in \S\ref{sec: crystals small}, including proofs of Theorems \ref{thm: char 3} and \ref{thm: char 4}. 
We conclude in \S\ref{sec: final} by giving two proofs of the aforementioned Corollary \ref{cor: number 2} and some final remarks. In Appendix \ref{sec: recover} we relate our constructions to those of \cite{OSSZ}.

\section{A literature review}
\subsection{Symmetric chain decompositions of the Young lattice.}
\label{sec: literature}
Stanley \cite{Stanley} conjectures 
that a solution to Problem \ref{prob: SCD} exists, observing that it is true 
for $n \le 2$. Unbeknownst to him at the time, solutions for all $n \le 4$ had been found by Rie{\ss} \cite{Riess} 
two years earlier. 
This marks a precedent in the area that soon becomes a tradition.
Solutions for $n\! =\! 3$ by Lindström \cite{Lindstrom} and $n\! =\! 4$ by West \cite{West} appeared shortly after, only acknowledging \cite{Riess} after the reviewing process. The constructions of Rie{\ss} do not coincide with those that came later; this too will be tradition.

Greene \cite{Greene} attributes to folklore that a greedy algorithm suffices to solve the problem, but shows that the approach is only successful for $n \le 4$. Greene's paper is significant, as it is the first attempt to solve the problem \emph{with a single construction}, in which $n$ is nothing more than a parameter. Again, the constructions are new and distinct.

Wen \cite{Wen} finds new computer-generated solutions to the problem for $n = 3$ and $4$ based on a modified greedy algorithm.
Dhand \cite{Dhand} creates a framework in the language of tropical geometry that produces solutions for $0 \le n \le 4$.

David, Spink, and Tiba \cite{DST} develop a geometric framework fitting the solutions of \cite{Lindstrom,Wen,West}. Embed $L(n,m)$ into $\R^n$ by treating  partitions as vectors. Then, dilate the embedding by $1/m$. For all $m$, the resulting set lies inside a fixed simplex $\Delta$ of $\R^n$, which is then divided into \emph{regions}, each carrying a \emph{direction}. A simultaneous solution for all $\{L(n,m) : m\ge 0\}$ is obtained from these regions, up to compatibility assumptions. Xin and Zhong \cite{Xin-Zhong} study the applicability of a greedy algorithm to the problem, rediscovering the work of Greene.

In 2024 we see an explosion in interest for the problem. Most importantly, Wen \cite{Wen2} manages to find computer-generated solutions to the $n\! =\! 5$ case: the first real progress after the conjecture was posed. However, their solution does not solve Problem \ref{prob: plethysm}.
Coggins, Donley, Gondal, and Krishna \cite{CDGK} reframe Lindström's construction in a diagrammatic way.

Orellana, Saliola, Schilling, and Zabrocki \cite{OSSZ} solve the problem for $n\! =\! 3$ and $4$. They attribute the abundance of different solutions to Problem \ref{prob: SCD} being too unrestricted, and impose further \emph{desirable properties} to the decompositions.
They deduce counting formulas for $\smash{a_{1^n[r]}^{k}}$ and recursive formulas for $\qbinom{r+1}{n}$ for $n\! =\! 3$ and~$4$.\smallskip

Our work maintains the tradition of rediscovery in the area. The bulk of the research was done independently of the authors above, and before some the mentioned works  were released.
Our framework has characteristics in common with several of the papers in the literature:
\begin{enumerate}
    \item The strategy is independent of $n$, as in \cite{Dhand, Greene, Xin-Zhong}.
    \item The framework is geometric in the sense of \cite{DST}.
    \item The constructions obtained are explicit in the sense of \cite{OSSZ} and satisfy all of the additional {desirable properties} ---some of which we do not impose a priori---, which allow us to show that the constructions are equivalent (Propositions \ref{p: recover 3} and \ref{p: recover 4}). This is the first instance that we know of two essentially different methods arriving to equivalent constructions.
\end{enumerate}
The constructions presented here have only been shown to solve Problem \ref{prob: SCD} for $n\le4$.
The graphs have been implemented in SageMath \cite{sage} and the decomposition of the character formulas into $q$-integers have been computed for $r \le 100$. This highlights that the algorithms are computationally efficient.

\subsection{Gaussian coefficients and the Young lattice}

A survey of Problems \ref{prob: plethysm} and \ref{prob: SCD} would not be complete without a mention of some related problems. The remaining two parts of the survey are less exhaustive. We follow \cite[Ch.~3]{EC1} for the background on posets.\medskip

A sequence $(a_0, a_1, \ldots, a_k)$ of integers is said to be \emph{unimodal} if there is an element $a_M$ such that $a_0 \le a_1 \le \cdots \le a_M \ge \cdots \ge a_k$, and \emph{symmetric} if $a_i = a_{k-i}$ for all $i$.
A polynomial $f = \sum_{d=0}^{k} a_d x^{d}$ is unimodal (resp.~symmetric) if the sequence of its coefficients $(a_{0}, a_{1}, \ldots, a_k)$ is unimodal (resp.~symmetric).
A poset $P$ is \emph{ranked} if it has a function $\textup{rk} : P \to \Z_{\ge0}$ such that if $x \gtrdot y$ then $\textup{rk}(y) = \textup{rk}(x) - 1$.
A ranked poset is \emph{rank-unimodal} (resp.~\emph{rank-symmetric}) if its rank-generating function $\sum_{x\in P} q^{\textup{rk}(x)}$ is unimodal (resp.~symmetric).
A ranked poset is \emph{Sperner} if the largest incomparable subset is the size of the widest rank-level.

\begin{fact}[\cite{Stanley}]
A rank-symmetric poset admitting a decomposition into rank-symmetric saturated chains is rank-unimodal and Sperner. 
\end{fact}

The rank-generating function of $L(n,m)$ is the polynomial in $q$ obtained from the centred $q$-binomial as $q^{nm/2}\qbinom{n+m}{n}$.
We refer to its coefficients as the \emph{Gaussian coefficients}. That Gaussian coefficients form a unimodal sequence follows immediately from the bijection \eqref{eq: bijection B(n) L(n)} between a basis of $\Alt^n\Sym^r\CC^2$ and $L(n,m)$  ---a homogeneous polynomial in $(x,y)$ specialised to $(1,q)$ is symmetric and unimodal if and only if it is the character of a representation of $\sl_2$ \cite{StanleySl2}---, but the result predates representation theory.
In the context of invariant theory and with different notation, the Gaussian coefficients were conjectured to form a unimodal sequence by Cayley in 1856, and proved for the first time by Sylvester in \cite{SylvesterUnimod}.
Gaussian coefficients are now known to be `strictly' unimodal, see \cite{PakPanova}.

The Sperner property of $L(n,m)$ was established in the aforementioned \cite{Stanley} using powerful tools from algebraic geometry. Stanley's proof was then simplified by Proctor \cite{Proctor}. \medskip

Another weakening of Problem \ref{prob: SCD} is to consider the reduced Young lattice $\tilde{L}(n,m)$, which is constructed from the Young lattice by adding more covering relations: all partitions of a given rank are covered by all partitions of the next rank. A decomposition of $\tilde{L}(n,m)$ into rank-symmetric saturated chains was found by O'Hara \cite{OHara, Zeilberger}, indirectly showing rank-unimodality of $L(n,m)$. In \cite{Zeilberger2} and \cite{Macdonald}, the method was used to give an elementary algebraic proof of the rank-unimodality of $L(n,m)$ which uses only  Newton's binomial theorem. O'Hara's method is recursive: the reduced Young lattice is partitioned into smaller posets, which are defined using a non-rank-symmetric chain decomposition.
These posets belong to a family in which each one is a Cartesian product of smaller posets of the same family.  One can construct a symmetric chain decomposition of a Cartesian product from symmetric chain decompositions of the factors. Solving a base case ends the proof.

We remark that the first non-rank-symmetric chains of O'Hara can be modified to decompose our original $L(n,m)$ \cite{Dhand2}, but the structure theorem no longer holds. This decomposition can still be used to establish the stronger strict unimodality of Gaussian coefficients purely combinatorially \cite{Dhand3},
but it does not solve Problem \ref{prob: SCD}.

\subsection{Plethysm and \texorpdfstring{$\sl_2$}{sl(2)}  isomorphisms}
The finite dimensional irreducible representations of $\sl_N$ are parametrised by partitions $\lambda$ of length $\le N$ and constructed with Schur functors as $S^\lambda \CC^N$. When $\lambda = (n)$ then $S^{(n)}V = \Sym^n V$, and when $\lambda = (1^n)$ then $S^{(1^n)}V = \Alt^n V$.
The most general problem of plethysm asks for the decomposition of $S^{\mu}S^{\nu}\CC^N$ into irreducible representations $S^\lambda\CC^N$ of $\sl_N$.

To make the problem more approachable, some parameters are specialised. For instance, $\Sym^2 S^\lambda\, \CC^N$ and $\Alt^2 S^\lambda\, \CC^N$ were decomposed by Carré and Leclerc \cite{CarreLeclerc}. 
We refer the reader \cite{mystery} for a brief survey of other known results. 
In this document, we take $N = 2$ and $\mu = (1^n)$ to obtain Problem \ref{prob: plethysm}. We solve it for $n \le 4$. We remark that solutions for $n\le 4$ can be found in the literature \cite{Thrall, Foulkes}, and that the cases $n\ge5$ were announced by Pak and Panova.
\medskip

There is a second way of simplifying the problem, by reducing the number of cases one needs to consider. The bijection we give in \eqref{eq: bijection B(n) L(n)} from a basis of $\Alt^n\Sym^r\CC^2$ to $L(n,m)$
lifts to an isomorphism of $\sl_2$ representations between $\Alt^n\Sym^r\CC^2$ and $\Sym^n\Sym^m\CC^2$.
This isomorphism is sometimes attributed to Murnaghan \cite{Murnaghan}, 
but it is known as the Wronskian isomorphism in invariant theory \cite{AC}. 
Since the word `isomorphism' already refers to a modern interpretation of the result, the origins are hard to trace ---Schmidt \cite{Schmidt} 
already considers the result to be classical in 1939.
The Wronskian isomorphism can be further lifted to an isomorphism of representations of $\SL_2(\FF)$ over any characteristic \cite{McDowellWildon}. The constructions we will present for $\Alt^n\Sym^r\CC^2$ can therefore be adapted to study $\Sym^n\Sym^m\CC^2$.

Similarly, it is apparent from our combinatorial approach that $m$ and $n$ play symmetric roles. An isomorphism between $\Sym^n\Sym^m\CC^2$ and $\Sym^m\Sym^n\CC^2$ was discovered independently by Hermite \cite{Hermite} and later Cayley \cite{Cayley} (with some special cases published simultaneously by Sylvester \cite{Sylvester}). It is now know as Hermite reciprocity or the Cayley--Sylvester formula, and was also lifted to an isomorphism of representations of $\SL_2(\FF)$ over any characteristic in \cite{McDowellWildon}.
See \cite{PagetWildon, MartinezWildon, hooks} for more reductions of this type, that allow one to understand more complicated $\sl_2$ plethysms in terms of $\Alt^n\Sym^r\CC^2$.

\section{Definitions}\label{sec: preliminaries}
Throughout the paper, we always assume $r+1 = m+n$.

\subsection{Plethystic tableaux}
We assume familiarity with the basic combinatorial objects of representation theory, and refer the reader to \cite[\S7.10]{EC2} for background on tableaux, and \cite[\S7.A2]{EC2} as a general reference for plethysm.\medskip

Let $\{x,y\}$ be a basis of $\CC^2$. A canonical basis of $\Sym^r\CC^2$ is given by 
\[\{x^r,~~ x^{r-1}y,~~ x^{r-2}y^2,~~ \ldots,~~ y^{r}\}.\]
In particular, note that $\Sym^r\CC^2$ is $(r+1)$-dimensional. Given a basis $\{v_0, \ldots, v_r\}$ of a vector space $V$, a canonical basis of $\Alt^n V$ is given by 
\[\{v_{a_n}\wedge \ldots\wedge v_{a_2} \wedge v_{a_1} \ : \  r\ge a_1 > a_2 > \cdots > a_n\ge 0\}.\]

Combinatorially, a basis of $\Sym^r\CC^2$ is labeled by the set $\SSYT_2(r)$ of semistandard Young tableaux of shape $(r)$ in two letters. We identify tableaux in $\SSYT_2(r)$ with bold integers $\boldsymbol{0}, \boldsymbol{1}, \ldots, \boldsymbol{r}$ by sending the tableau $\rowTab{1&\cdots&1&2&\cdots&2}$ with $a$ twos to the number $\boldsymbol{a}$. For instance, for $r\! =\! 9$,
\[
\boldsymbol{3} 
~~
\text{ stands for }
~~
{\rowTab{1&1&1&1&1&1&2&2&2}}
~~
\text{ which stands for }
~~
x^6y^3.
\]

A canonical basis of $\Alt^n\Sym^r\CC^2$ is then labeled by the set $\SSYT_{2}(1^n[r])$ of \emph{plethystic $\sl_2$ tableau} of outer shape $(1^n)$ and inner shape $(r)$, which are semistandard Young tableaux of shape $(1^n)$ with entries in $\SSYT_2(r)$. By the above identification, $\SSYT_{2}(1^n[r])$ is identified with the set $\SSYT_{[\boldsymbol{0},\boldsymbol{r}]}(1^n)$ of tableaux whose entries we write in bold. To save space, we usually write $\colTabInline{a_n,\cdots,a_1}$ for a column tableau. Summing up in an example: for $n\! =\! 2$ and $r\! =\! 5$,
\[
\colTabInline{1,3} 
~~
\text{ stands for }
~~
\colTab{1,3}
~~
\text{ which stands for }
~~
\def\arraystretch{1.2}
\begin{array}[m]{|l|}
\hline
{\rowTab{1&1&1&1&2}} \\ \hline
{\rowTab{1&1&2&2&2}} \\ \hline
\end{array}
~~
\text{ which stands for }
~~
x^4y\wedge x^2y^3\, .
\]
Recall $r+1=m+n$ and $L(n,m) = \{\lambda \ :\ \lambda_1 \le n,~ \ell(\lambda)\le m \}$. The map
\begin{align}\label{eq: bijection B(n) L(n)}
    \Psi: \SSYT_2(1^n[r]) &\to L(n,m) \notag\\
    \colTabInline{a_n, \cdots, a_1} &\mapsto (n^{a_n}\ldots i^{a_i - a_{i+1}-1} \ldots 1^{a_1-a_2-1})
\end{align}
is a bijection.
It is better understood via an example: let $n\! =\! 3$ and $r\! =\! 6$, then
\begin{equation*}
    \colTab{0,3,5} = 
    \def\arraystretch{1.2}
    \begin{array}[m]{|l|}
    \hline 
    \rowTab{1&1&\cellcolor{yellow!50}1&\cellcolor{yellow!50}1&\cellcolor{yellow!50}1&\cellcolor{yellow!50}1} \\ \hline
    \rowTab{1&\cellcolor{yellow!50}1&\cellcolor{yellow!50}1&\cellcolor{blue!20}2&\cellcolor{blue!20}2&2} \\ \hline
    \rowTab{\cellcolor{yellow!50}1&\cellcolor{blue!20}2&\cellcolor{blue!20}2&\cellcolor{blue!20}2&2&2} \\ \hline
    \end{array}\ \mapsto
    {\ytableausetup{boxsize=10pt}\fontsize{8}{9}\selectfont
    \ytableaushort{1111,1122,1222}* [*(yellow!50)]{4,2,1}* [*(blue!20)]{0,2+2,1+3}}\ \mapsto
    {\ytableausetup{boxsize=10pt}\fontsize{8}{9}\selectfont
    \ytableaushort{221,221,211,111}* [*(yellow!50)]{2+1,2+1,1+2,3}* [*(blue!20)]{2,2,1}} \mapsto
    \ydiagram[*(blue!20)]{2,2,1}\,.
\end{equation*}
We work with tableaux in $\SSYT_2(1^n[r])$ as our central object, which the reader may identify with either basis elements of $\Alt^n\Sym^r\CC^2$ or with partitions of $L(n,m)$.

\begin{note}
    The set $\SSYT_2(1^n[r])$ is naturally in bijection with $\SSYT_{[0,r]}(1^n)$, by sending a tableau with bold entries to the same tableau with non-bold entries. We note that $\SSYT_{[0,r]}(1^n)$ has a $\sl_{r+1}$ crystal structure, corresponding to the representation $\Lambda^{n}\CC^{r+1}$  \cite[\S3]{BS}. 
    We thus choose to work with plethystic $\sl_2$ tableau since they are already familiar to crystal theorists.
    Finding the $\sl_2$ crystal for $\Lambda^n\Sym^r\CC^2$ as a refinement of the $\sl_{r+1}$ crystal for $\Lambda^n\CC^{r+1}$ is morally the approach of \cite{OSSZ}.
\end{note}

\subsection{Crystals and characters}\label{sec: crystals and characters}

For background on crystals as combinatorial objects, we refer to \cite{BS}. In this paper a crystal is a seminormal $\sl_2$ crystal. The next definition is simplified accordingly.
\begin{definition} \label{de: crystal}
    A \emph{crystal} is a set $\B$ together with two maps $F : \B \to \B \cup\{0\}$ and $\wt : \B \to \Z$ such that:
    \begin{axioms}{C} 
    \setcounter{enumi}{-1}
    \item\label{C0} $F$ restricted to $F^{-1}(\B)$ is injective,
    \item\label{C1} $\varphi(b) = \varepsilon(b) + 2\wt(b)$ for all $b \in \B$, and
    \item\label{C2} $\wt(F.\,b) = \wt(b) - 1$ whenever $F.\,b \ne 0$,
    \end{axioms}
    where $\varphi(b) = \max\{k ~:~ F^k.\,b \ne 0\}$ and $\varepsilon(b) = \max\{k ~:~ E^k.\,b\ne0\}$, and $E : \B \to \B\cup\{0\}$ is the inverse of $F$ where the inverse is defined, and otherwise $0$. 
\end{definition}
We identify a crystal with the directed graph with vertex set $\B$ and an arc $x\longrightarrow y$ whenever $F.\,x = y$. Although not immediately clear, the weight $\wt$ can be recovered from the graph \cite[Lemma 2.14]{BS}. By \ref{C0} the graph is a disjoint union of paths, by \ref{C1} the graph is weight-symmetric.

Let $\alpha$ be the simple root of $\sl_2$, so that $\Z\frac{\alpha}{2}$ is its weight lattice. A \emph{crystal of a representation} $V$ is a crystal on the set of weights of $V$ (with multiplicity) such that $\wt(k\frac{\alpha}{2}) = k/2$.

A vertex $b$ of a crystal is a \emph{highest weight element} if $E\,.\,b=0$. There is exactly one highest weight element $b$ for each connected component of a crystal of $V$, and the connected component of $b$ is the crystal of $\Sym^{2\wt(b)}\CC^2$. Therefore, given the set $\HW$ of highest weight elements of the crystal of $V$, we can write
\begin{equation}\label{eq: V and hw}
V \cong \bigoplus_{b\in\HW} \Sym^{2\wt(b)}\CC^2.
\end{equation}

Characters of $\sl_2$ representations are homogeneous polynomials in two variables $x, y$. Since the trace of an $\sl_2$ matrix is $0$, we may specialise these two variables to $\smash{q^{-{1}/{2}}}$ and $\smash{q^{{1}/{2}}}$ without loosing any information. 
The character of the representation
is retrieved from its crystal via $\smash{\sum_{b\in\B} q^{\wt(b)}}$ \cite[\S 2.6 and Ch.~13]{BS}.
The character of $\Sym^{r}\CC^2$ is the Schur polynomial
\[
s_{(r)}(q^{-\frac{1}{2}},q^{\frac{1}{2}}) = {[r+1] = \frac{q^{\frac{r+1}{2}}-q^{-\frac{r+1}{2}}}{q^{\frac{1}{2}}-q^{-\frac{1}{2}}}} = q^{-\frac{r}{2}} + q^{1-\frac{r}{2}} + q^{2-\frac{r}{2}} + \cdots + q^{\frac{r}{2}},
\]
and the character of $\Alt^n\Sym^{r}\CC^2$ is the plethysm of Schur polynomials
\[
s_{(1^n)}\circ s_{(r)}(q^{-\frac{1}{2}},q^{\frac{1}{2}}) = {\qbinom{r+1}{n} = \frac{[r+1][r]\cdots[r-n+2]}{[n][n-1]\cdots[1]}}.
\]
We will prefer the notation in terms of $q$-integers and $q$-binomials hereafter.

\begin{note}
    Our definition of $q$-integers (and hence $q$-binomials) is non-standard. There are two objects typically called by this name: combinatorialists define the \emph{$q$-analogue of the integer} $r+1$ to be
    \[
    \frac{1-q^{r+1}}{1-q} = 1 + q + \cdots + q^{r}. 
    \]
    With our notation, this is $q^{r/2}[r+1]$.
    Algebraists define the \emph{quantum integer} $r+1$ to be
    \[
    \frac{q^{r+1}-q^{-r-1}}{q-q^{-1}} =  q^{-r} + q^{-r+2} + \cdots + q^{r-2} + q^{r}. 
    \]
    With our notation, this is $[r+1]|_{q\mapsto q^2}$.
    The $q$-integers we define arise most naturally in this paper.
\end{note}

\subsection{Posets}\label{sec: preliminaries posets}

Many of the concepts that we use can be understood purely in the language of posets. We follow \cite[Ch.~3]{EC1} for the terminology.

Let $P$ be a ranked poset $P$  and let $\textup{rk} : P \to \Z_{\ge0}$ be its rank function.
A \emph{chain} is a poset in which the order is total.
A \emph{saturated chain} of a ranked poset is a subposet that is a chain and that `does not skip any rank': if it has an element of rank $k$ and an element of rank $k+N$, then it has an element of each rank $k + i$ for $0<i<N$.
A \emph{symmetric chain decomposition} of a ranked poset is a decomposition of the poset into rank-symmetric, saturated chains.
We describe symmetric chain decompositions in terms of a function matching each rank and the next one.\medskip

The bijection $\Psi$ of \eqref{eq: bijection B(n) L(n)} sends a tableau with entry-sum $k$ to a partition of size $k - \frac{(n-1)(n-2)}{2}$. The weight function on $\SSYT_2(1^n[r])$ will be a linear function on the entry-sum of the tableaux; the size of a partition is its rank on $L(n,m)$. Hence the weight function of our crystals will be the rank function on $L(n,m)$ up to some shift. 
The character of a crystal of $\Alt^n\Sym^r\CC^2$ is $\qbinom{r+1}{n} = \qbinom{n+m}{n}$ and the rank-generating function of $L(n,m)$ is $q^{nm/2}\qbinom{n+m}{n}$.

\section{Crystals of  \texorpdfstring{$\Alt^n\Sym^r\CC^2$}{Alt\^{}n Sym\^{}r C\^{}2}}\label{sec: crystals n}

Set $\B_{r}(n) = \SSYT_2(1^n[r])$, and note that $\B_n(n) \subseteq \B_{n+1}(n) \subseteq \cdots$. Let $\B(n) = \bigcup_{r\ge n} \B_{r}(n)$. Define weight functions on each $\B_{r}(n)$ by 
\begin{equation}\label{eq: def of wt^r}
\wt_{r} \colTabInline{a_n,\cdots,a_1} = \frac{1}{2}\Big( nr - 2(a_n + \cdots + a_1) \Big).    
\end{equation}
We outline a program to find a map $F : \B(n) \to \B(n)$ for each $n$, such that restricting to each $\B_{r}(n)$ produces a crystal. Note that $F$ will depend on $n$.

Conceptually, we look for two operators $F^\top$ and $F^\bot$ from $\B(n)$ to itself, each of which satisfies axioms \ref{C0} and \ref{C2} above when restricted to $\smash{\B_{r}(n)}$, but not \ref{C1}. Assuming some compatibility properties (Problems \ref{prob: prob1}, \ref{prob: prob2}, and \ref{prob: prob3} below), the two operators can be glued together into an operator $F$ satisfying \ref{C1} when restricted to each $\B_{r}(n)$. See Figure \ref{fig: strategy}.
\begin{figure}[h]
    \centering
    \begin{subfigure}{.3\textwidth}
    \centering
    \includegraphics[]{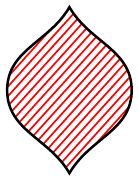}
        \subcaption{$F^\top$ in $\B_{r}(n).$}
    \end{subfigure}
    \begin{subfigure}{.3\textwidth}
    \centering
        \includegraphics[]{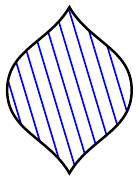}
        \subcaption{$F^\bot$ in $\B_{r}(n).$}
    \end{subfigure}
    \begin{subfigure}{.3\textwidth}
    \centering
    \includegraphics[]{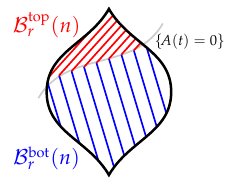}
        \subcaption{$F$ in $\B_{r}(n).$}
    \end{subfigure}
    \caption{$F$ is a function by parts solving Problem \ref{prob: SCD}.}
    \label{fig: strategy}
\end{figure}

The operator $F^\top$ will be defined inductively using the operator $F$ on $\B(n-2)$. 
The base cases of this induction are $F.\emptyset = 0$ on $\B(0)$ 
and $F.\,\pletRowTab{a} = \pletRowTab{a+1}$ on $\B(1)$.

\subsection{Top operator}

Let $t = \colTabInline{a_n, \cdots, a_1}\in\B_{r}(n)$.
By removing the first and last entry, we get a tableau \(t^\downarrow\) in \(\smash{\B(n-2)}.\)
Assuming by induction that we have a crystal structure on $\B_{k}(n-2)$ for all $k$, we can consider $F\,.\,t^\downarrow$. Adding back the removed entries, we obtain a new tableau in $\smash{\B_{r}(n)}$, which we define to be $F^\top\,.\,t$. More precisely,
\[
F^\top\,.\,\colTab{a_{n}, a_{n-1}, \cdots, a_{2}, a_1} = \colTab{a_n,b_{n-2} + (a_n+1),\cdots,b_1 + (a_n+1),a_1}\,, \quad\text{where}\quad
\colTab{b_{n-2},\cdots,b_1} = F.\,\colTab{a_{n-1} - (a_n + 1), \cdots, a_{2} - (a_n + 1)}\,.
\]
If $b_1 + (a_n + 1) = a_1$ then the resulting tableau is not strictly increasing and we let $F^\top\,.\,t = 0$.
\begin{example}\label{eg: Ftop n=2}
    For $n=2$, we have $t^\downarrow = \emptyset$ for all $t\in\B(2)$ and $F.\emptyset=0$. By convention, we set $F^\top.t = 0$.
\end{example}
\begin{example}
    Let $n=4$ and $t = \colTabInline{1,4,5,8}$. We begin by removing the first and last entry and renormalising the tableau by $a_n + 1 = 2$ as in the displayed equation, to get $t^\downarrow = \colTabInline{2,3}$. It will turn out that $F\,.\,\colTabInline{2,3} = \colTabInline{2,4}$ (Examples \ref{eg: Fbot n=2} and \ref{eg: F n=2}). Hence, renormalising and adding back the removed entries,
    we obtain $F^\top.\,t = \colTabInline{1,4,6,8}$.
\end{example}
\begin{lemma}
    The operator $F^\top$ satisfies axioms \ref{C0} and \ref{C2}.
\end{lemma}
\begin{proof}
    The operator is injective on $F^{-1}(\B(n))$ by induction. The image differs from the original tableaux exactly by one entry, also by induction.
\end{proof}

We define $E^\top : \B(n) \to \B(n)\cup\{0\}$ as the inverse of $F^\top$ where the inverse is defined, and otherwise $0$; let $\varepsilon^\top (t) = \max\{k \ : \ (E^\top)^k.\,t\ne0\} = \varepsilon(t^\downarrow)$.

\subsection{Bottom operator}
\label{sec: Fbot}
Embedding $\B(n)$ into $\R^n$ via $\colTabInline{a_n, \cdots, a_1} \mapsto (a_n, \ldots, a_1)'$ allows us to talk about directions. We look for a decomposition of $\B(n)$ into paths `parallel' to the vector $(1,\ldots,1)'$. Each path is of the form $P(t_0,v_{t_0}) = \{t_k\}_{k\ge0}$, where
\[
t_k = t_0 + \floor\left(\frac{1}{n}v_{t_0} + \frac{k}{n}\Vector{1, \vphantom{\int\limits^x}\smash{\vdots},1}\right)
\]
for some \emph{initial tableau} $t_0$ and some \emph{offset vector} $v_{t_0} \in \mathbb{S}_n.(0,1,\ldots,n-1)'$. The floor is taken entry-wise. We think of $P(t_0,v_{t_0})$ as a {discretisation} of the line $t_0 + \langle(1,\ldots,1)'\rangle$.
We look for a set of pairs $(t_0,v_{t_0})$ that we call a \emph{seed}.

\begin{example}\label{eg: P(t,o) n=2}
    Let $n=2$ and consider the initial tableau $t_0 = \colTabInline{0,3}$ with offset vector $v_{t_0} = (0,1)'$. Then, the path initiated at $t_0$ is
    \[
    P(t_0, v_{t_0}) = \Big\{\ \colTab{0,3}\,,~ \colTab{0,4}\,,~ \colTab{1,4}\,,~ \colTab{1,5}\,,~ \colTab{2,5}\,,~ \ldots \Big\}.
    \]
\end{example}

\begin{prob}\label{prob: prob1}
    Find a seed $S$ such that $\{P(t_0,v_{t_0})  :  ({t_0},v_{t_0})\!\in\! S\}$ is a set partition of $\B(n)$.
\end{prob}
\begin{example}\label{eg: S(n) n=2} For $n = 2$, the seed 
\[
S(2) = \Big\{\Big(~\colTab{0,a_0}~,~{\scriptsize\begin{pmatrix}0\\1\end{pmatrix}}~\Big) \ : \ a_0 \text{~odd}\Big\}
\]
is a solution to Problem \ref{prob: prob1}.
Indeed, note that  a tableau is in $P(\,\colTabInline{0,a_0},(0,1)')$ if and only if it is of the form
    \[
    \colTab{i,a_0+i}
    \quad\text{or}\quad
    \colTab{i,a_0+i+1}
    \]

    \vskip.5em
    \noindent
    for some $i$.
    Equivalently, a tableau $\colTabInline{b,a}$ is in one of the paths if either
    \begin{enumerate}[label=(\roman*)]
        \item $a - b \equiv 0 ~ (2)$, or
        \item $a - b \equiv 1 ~ (2)$.
    \end{enumerate}
    These sets are disjoint and partition $\B(2)$. We will use this seed in our running example.
\end{example}

Assuming a solution to Problem \ref{prob: prob1}, we can define $F^\bot$ via $F^\bot\,.\,t_k = t_{k+1}$ for each element $t_k$ of a path $P(t_0, v_{t_0})$. Hence $F^\bot$ is defined on $\B(n)$.

\begin{lemma}
    The operator $F^\bot$ restricted to $\B_r(n)$ satisfies axioms \ref{C0} and \ref{C2}.
\end{lemma}
\begin{proof}
    The operator is injective-where-defined, since every tableau belongs to exactly one path. The tableaux in a path change exactly by one entry each step, since the offset is in the orbit $\mathbb{S}_n.(0,1,\ldots,n-1)'$.
\end{proof}
\begin{example}\label{eg: Fbot n=2}
    For $n=2$, we have
    \[
            F^\bot.\,\colTab{b, a} = \begin{cases}
            \colTabInline{b+1,a} & \text{if } a\equiv b \, (2),\\            
            \colTabInline{b, a+1} & \text{if } a\not\equiv b \, (2).
        \end{cases}
    \]
\end{example}

We define a map $\varphi^\bot_{r} : \B_r(n) \to \Z_{\ge0}$ which maps $t_k$ to $K-k-1$, where $t_K$ is the first element of $P(t_0,v_{t_0})$ lying in $\B_{r+1}(n)$.

Let $t$ be a tableau and $a_1$ be its largest entry. The next lemma shows that $\varphi^\bot_r(t)$ is approximately $n(r-a_1)$, after noting that $0 \le \varphi^\bot_{a_1}(t) < n$.
\begin{lemma}
\label{note: a formula for varphi bot}
    Let $t = \colTabInline{a_n, \cdots, a_1}\in\B(n)$ be a tableau. Then,
    $\varphi^\bot_{r}(t) = n(r - a_1) + \varphi^\bot_{a_1}(t).$
\end{lemma}
\begin{proof}
    Fix a tableau $t = \colTabInline{a_n, \cdots, a_1}\in\B(n)$.
    Apply $F^\bot$ successively to obtain a sequence
    \[
    t, \quad (F^\bot).\,t, \quad (F^\bot)^2.\,t, \quad \ldots, \quad (F^\bot)^{n-1}.\,t, \quad (F^\bot)^n.\,t. 
    \]
    In each step, exactly one entry changes between a tableau $(F^\bot)^{k}.\,t$ and the next $(F^\bot)^{k+1}.\,t$.
    After $n$ steps, we have $(F^\bot)^n.\,t = t+(1,1,\ldots,1)'$, and therefore there is a unique step \break$0\le k< n$ such that 
    \[
    (F^\bot)^{k+1}.\,t = (F^\bot)^{k}.\,t + (0,\ldots,0,1)'.
    \]
    Note that $k = \varphi^\bot_{a_1}(t)$.
    We have $\varphi^\bot_{r}((F^\bot)^k.\,t) = n(r - a_1)$, and thus
    \[
    \varphi^\bot_{r}(t) = n(r - a_1) + \varphi^\bot_{a_1}(t).\qedhere
    \]
\end{proof}
\begin{example}\label{eg: varphibot n=2}
    Let $n=2$ and $t = \colTabInline{b,a}$. We have
    $\varphi^\bot_{r}(t) = 2(r-a) + \varphi^\bot_{a}(t).$
    By inspection of the formula from Example \ref{eg: Ftop n=2}, we get $F^\bot.\,t \in \B_{a+1}(2)$ only if $a \not\equiv b~(2)$. That is, $\varphi^\bot_{a}(t) = \delta_{a\equiv b(2)}$. Therefore, $\varphi^\bot_{r}(t) = 2(r-a) + \delta_{a\equiv b(2)}$.
\end{example}

\subsection{Gluing up}
Recall the axiom \ref{C1}, which we rewrite as a functional equation
\begin{enumerate}
    \item[\ref{C1}] $\smash{\varphi - (\varepsilon + 2\wt) = 0}$.
\end{enumerate}
It is not in general satisfied by $F^\bot$ nor $F^\top$. 
Define a map $A: \B(n) \to \Z$ by
\begin{equation}
    \label{eq: A}
A = \smash{\varphi^\bot_{r}} - \big(
\varepsilon^\top + 2\wt_{r}
\big).
\end{equation} 
The next lemma shows that $A$ does not depend on $r$.
\begin{lemma}
\label{note: formula for A(t)}
    Let $t = \colTabInline{a_n, \cdots, a_1}$. We have
    $A(t) = \varphi^\bot_{a_1}(t) - \varphi_{a_1-a_n-1}(t^\downarrow) + na_n.$
\end{lemma}
\begin{proof}
Gathering our formula for $\wt_r$ \eqref{eq: def of wt^r}, together with the formula for $\varphi^\bot_{r}$ from Lemma \ref{note: a formula for varphi bot}, and axiom \ref{C1} for crystals on $\B_{a_1-a_n-1}(n-2)$, we can compute $A(t)$ to be
    \begin{align*}
        A(t) &= \varphi^\bot_{r}(t) - \big(\varepsilon^\top(t)+ 2\wt_{r}(t)\big)\\
        &= n(r-a_1) + \varphi^\bot_{a_1}(t)\\
        &~~~ - \varphi_{a_1-a_n-1}(t^\downarrow) + n(a_1+a_n) - 2(a_n+\cdots+a_1)\\
        &~~~ - nr + 2(a_n + \cdots + a_1)\\
        &= \varphi^\bot_{a_1}(t) - \varphi_{a_1-a_n-1}(t^\downarrow) + na_n. \qedhere
    \end{align*}
\end{proof}

\begin{example}\label{eg: A(t) n=2}
    Let $n=2$ and $t = \colTabInline{b,a}$. By Example \ref{eg: Fbot n=2} we have $t^\downarrow = \emptyset$ and so $\varphi_{r}(t^\downarrow) = 0$ for all $r$.
    Using the expression for $\varphi^\bot_a(t)$ from Example \ref{eg: varphibot n=2}, we obtain
    \(
    A(t) = 2b + \delta_{a\equiv b(2)}.
    \)
\end{example}

As in Figure \ref{fig: strategy}, we consider 
\[
\Top(n) = \{t\in\B(n) \ : \ A(t)<0\}
\quad \text{and}\quad
\Bottom(n) = \{t\in\B(n) \ : \ A(t)\ge0\}.
\]
\begin{definition}\label{de: crystal n}
    Fix $n\in\Z_{\ge0}$.
     Define an operator $\B(n) \to \B(n)$ by
     \[
     F.\,t = 
     \begin{cases}
        F^\top\,.\,t &\text{if }t\in\Top(n),\\
        F^\bot\,.\,t &\text{if }t\in\Bottom(n).
        \end{cases}
     \]
\end{definition}

\begin{example}\label{eg: F n=2}
For $n=2$ we have $\Top(2) = \emptyset$. Therefore $F = F^\bot : \B(2) \to \B(2)$. A formula for $F^\bot$ was given in Example \ref{eg: Fbot n=2}.
\end{example}

\begin{note}
    The previous definition is reminiscent of Kashiwara's tensor product rule \cite[\S2.3]{BS}, but we can make this more precise.
    Let $\mathcal{B}_1 = (\mathcal{B}_1, F_1, \wt_1)$ and $\mathcal{B}_2 = (\mathcal{B}_2, F_2, \wt_2)$ be two $\sl_2$ crystals. Let $\varphi_i$ and $\varepsilon_i$ be defined in the usual manner, for $i=1,2$.

    Our objective is to construct a crystal structure on $\mathcal{B}_1\times \mathcal{B}_2$ with weight function $\wt(a, b) = \wt_1(a) + \wt_2(b)$.
    As a first attempt, define an operator 
    $(F_1\otimes1).(a, b) = (F_1.a, b)$. This operator satisfies axioms \ref{C0} and \ref{C2}, but not necessarily \ref{C1}. Similarly, we could attempt to define $1\otimes F_2$ similarly, but we arrive at the same conclusion. 
    
    Kashiwara's solution in
    \cite[Prop.~6]{Kashiwara}
    is to combine $F_1\otimes1$ and $1\otimes F_2$ via
     \[
     F.\,(a, b) := 
     \begin{cases}
        (F_1\otimes 1).(a, b) &\text{if }\varphi_2(b) > \varepsilon_1(a),\\
        (1\otimes F_2).(a, b) &\text{if }\varphi_2(b) \le \varepsilon_1(a).
        \end{cases}
     \]
     Working with $\sl_2$ crystals, this is the Clebsch--Gordan rule, but Kashiwara's rule generalises it to any crystal.
     Now note that applying axiom \ref{C1} for $F_1$ and $F_2$ gives
     \begin{align*}
         \varphi_2(b) > \varepsilon_1(a)
         & \iff 
         \varepsilon_2(b) + 2\wt_2(b)  > \varphi_1(a) - 2\wt_1(a)
         \\ & \iff 
         \varepsilon_2(b) + 2\wt(a, b)  > \varphi_1(a) &\iff
         \tilde{A}(a, b) < 0,
     \end{align*}
     where $\tilde{A} := \varphi_1 - (\varepsilon_2 + 2\wt)$. We have recovered the same expression found in Definition \ref{de: crystal n}.
\end{note}

Define $E : \B(n)\to\B(n)\cup\{0\}$ be the inverse of $F$ if it exists, and $0$ otherwise; define 
$\varepsilon(b) = \max\{k \ : \ E^k.\,b \ne 0\}$ as before.
For each $r$, the operator $F$ restricts to an operator $F_r : \B_r(n)\to\B_r(n)\cup\{0\}$ by letting $F_r\,.\,t = 0$ if $F.\,t \not\in \B_r(n)$.
Define $\varphi_r(b) = \max\{k \ : \ (F_r)^k.\,b\ne0\}$. An tableau is \emph{lowest weight} if $F_r\,.\,t = 0$. We need a lemma relating lowest weight tableaux and the bottom $\Bottom(n)$.
\begin{lemma}\label{lem: Ft=0 in bot}
    If $F_r\,.\,t = 0$ then $A(t) \ge 0$.
\end{lemma}
\begin{proof}  
    Let $t = \colTabInline{a_n, \cdots, a_1} \in \B(n)$ and suppose $F^\top\,.\,t = F_r\,.\,(t^\downarrow) = 0$. We thus deduce
    \(\varphi_{a_1-a_n-1}(t^\downarrow)=0\) and get
    \[
        A(t) = \varphi^\bot_{a_1}(t) - \varphi_{a_1-a_n-1}(t^\downarrow) + na_n= \varphi^\bot_{a_1}(t) + na_n \ge 0. \qedhere
    \]
\end{proof}

For highest weight tableaux, we can similarly show that $E\,.\,t = 0$ implies $A(t)\le1$, but we need something stronger for Theorem \ref{thm: MAIN} below.
\begin{prob}\label{prob: prob2}
    Find a seed such that if $E\,.\,t=0$ then $A(t)\le0$.
\end{prob}

\noindent Finally, we need to control what happens around the points for which $A(t) = 0$.

\begin{prob}\label{prob: prob3}
    Find a seed such that if $A(t)\!<\!0\!\le\!A(F.\,t)$ then
    $A(F.\,t)\!=\!0\!<\!A(F^k.\,t)$ for all $k\!\ge\!2$.    
\end{prob}

\begin{example}\label{eg: S(2) probs23 n=2}
    The seed $S(2)$ from Example \ref{eg: S(n) n=2} is a  solution to Problems \ref{prob: prob2} and \ref{prob: prob3}. Indeed, let $t = \colTabInline{b,a} \in \B(2)$ be a tableau and suppose $E_\bot\,.\,t = 0$. Then $t$ is initial by Example \ref{eg: F n=2}, and the formula for $A(t)$ from Example \ref{eg: A(t) n=2} gives $A(t) = 0 + 0 = 0$. On the other hand, the condition to Problem \ref{prob: prob3} is void, as there is no tableau such that $A(t)<0$.
\end{example}

\begin{theorem}\label{thm: MAIN}
    If a seed is a solution to Problems \ref{prob: prob1}, \ref{prob: prob2}, and \ref{prob: prob3} then the restriction $F_r$ of $F$ to each $\B_{r}(n)$ defines a crystal.
\end{theorem}
\begin{proof}
    We check that axioms \ref{C0}, \ref{C1}, and \ref{C2} of Definition \ref{de: crystal} are satisfied by $F_r$. Axiom \ref{C2} is clear by construction, since it holds for both $F^\top$ and $F^\bot$.
    Axiom \ref{C0} follows from the seed being a solution to Problem \ref{prob: prob3}:
    assume there exist tableaux $t\in\Top(n)$, $t'\in\Bottom(n)$ and $t''\in\B(n)$ such that
    \[
    F_r\,.\,t = F^\top\,.\,t = t'' =
    F^\bot\,.\,t' = F_r\,.\,t'.
    \]
    Then $A(t')\ge0$, which implies $A(t'')>0$ by Problem \ref{prob: prob3}. But $A(t)<0$ and hence by Problem \ref{prob: prob3} again, $A(t'') = 0$, which gives a contradiction.
    We thus have a disjoint union of paths.

    It remains to check \ref{C1}. Note that it suffices to check the axiom for one tableau on each path.
    Let $t$ be such that $E\,.\,t = 0$, and consider the sequence $\{F_r^k\,.\,t\}_{k\ge0}$ obtained by repeated applications of $F_r$. By Problem \ref{prob: prob2}, either $A(t)=0$ or $t\in\smash\Topr(n)$. Since $F_r^{N+1}.\,t = 0$ implies $A(F^N.\,t) \ge 0$ by Lemma \ref{lem: Ft=0 in bot}, the path eventually enters $\smash\Bottomr(n)$, say at the $k$th step for the first time. The tableau $F^k.\,t$ satisfies
    \[
    \varepsilon(F^k.\,t) = \varepsilon^\top(F^k.\,t),
    \quad \text{and} \quad
    \varphi_{r}(F^k.\,t) = \varphi^\bot_{r}(F^k.\,t).
    \]
    By Problem \ref{prob: prob3} we deduce $A(F^k.\,t) = 0$, which for $F^k.\,t$ is precisely axiom \ref{C1}.
\end{proof}

We end with a natural conjecture, which we show for $n\le 4$ in the following section.
\begin{conjecture}\label{conjecture}
    For all $n\ge0$, there exists a seed $S(n)$ which is a solution to Problems \ref{prob: prob1}, \ref{prob: prob2}, and \ref{prob: prob3}, and therefore Definition \ref{de: crystal n} ---which depends on $S(n)$--- defines a crystal operator on $\SSYT_2(1^n[r])$ for all $r \ge n$.
\end{conjecture}

\section{Explicit constructions for small \texorpdfstring{$n$}{n}}\label{sec: crystals small}
\subsection{Crystals of \texorpdfstring{$\Alt^2\Sym^r\CC^2$}{Alt\^{}2 Sym\^{}r C\^{}2}}

The running Example \ref{eg: Ftop n=2} -- \ref{eg: S(2) probs23 n=2} and Theorem \ref{thm: MAIN} amount to a proof of the next proposition.
\begin{proposition}
    The operator $F$ from Examples \ref{eg: Fbot n=2} and \ref{eg: F n=2} defines a crystal operator when restricted to each $\B_r(2)$.
\end{proposition}
In Figure \ref{subfig: crystal 2} we illustrate the corresponding crystals as embedded in the plane, writing $\colTabInline{ab}$ for $\colTabInline{a,b}$ to avoid cluttering. Note that the paths follow the $(1,1)'$ direction (South). Figure \ref{subfig: L(2,3)} is the image of Figure \ref{subfig: B4(2)} under the bijection $\Psi$ of \eqref{eq: bijection B(n) L(n)}.

\begin{figure}[ht]
    \begin{subfigure}[t]{0.32\textwidth} \centering
    \includegraphics[scale=1.5]{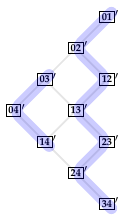}
    \subcaption{Crystal on $\B_4(2)$.}
    \label{subfig: B4(2)}
    \end{subfigure}
    \begin{subfigure}[t]{0.32\textwidth} \centering
    \includegraphics[scale=1.5]{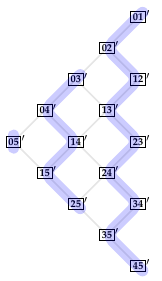}
    \subcaption{Crystal on $\B_5(2)$.}
    \label{subfig: B5(2)}
    \end{subfigure}
    \begin{subfigure}[t]{0.32\textwidth} \centering
    \includegraphics[scale=1.5]{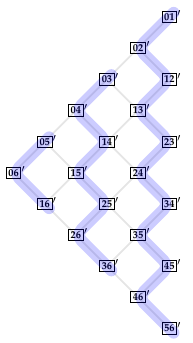}
    \subcaption{Crystal on $\B_6(2)$.}
    \end{subfigure}
    \caption{Crystal structures on $\B_r(2)$ as embedded in $\R^2$.}
    \label{subfig: crystal 2}
\end{figure}




With a solution of Problem \ref{prob: SCD} in hand, we now translate it to a solution of Problem \ref{prob: plethysm}. By~\eqref{eq: V and hw} we simply need to compute the set of highest weight tableaux.

The highest weight tableaux are by definition those for which $E\,.\,t = 0$, which by construction are simply $\HW(2) = 
\{\colTabInline{b,a}\ : \ \substack{b=0,\\ a~\text{odd}}\}$. Using \eqref{eq: V and hw} and comparing it with \eqref{eq: plethystic coefficients},
\begin{align*}
    a_{1^2[r]}^k &= \#\{t \in \HW(2)\cap\B_r(2) \ : \ a+2b = k\}\\
    &= \#\Big\{\colTabInline{0,a}\in\HW(2) :
    \substack{a + 2\cdot0 = k,\\ a\le r}\Big\} =  \delta_{k~\text{odd}} \cdot\delta_{k\le r} \cdot k.
\end{align*}
To finish, we obtain a recursive formula for the character of $\Alt^2\Sym^{r}\CC^2$.
Note that $\B_{r-1}(2) \subseteq \B_{r}(2)$, and that the crystal of $\B_{r}(2)$ is obtained by extending each connected component of that of $\B_{r-1}(2)$, and then adding one extra component if $r$ is odd. We have shown that the character of $\Alt^2\Sym^r\CC^2$ satisfies the recursion
\begin{equation}\label{eq: recursion n=2}
\qbinom{r+1}{2} = \qbinom{r}{2}_{+2} + \delta_{r~\text{odd}}\cdot [1].    
\end{equation}
Although not difficult, the $n\!=\!2$ case perfectly illustrates the nature of our constructions.

\subsection{Crystals of \texorpdfstring{$\Alt^3\Sym^r\CC^2$}{Alt\^{}3 Sym\^{}r C\^{}2}}
Set $\{\colTabInline{c,b,a} \ : \ c = 0,~b~\text{odd},~a-b\not\equiv2~(3)\}$ to be the set of \emph{initial} tableaux.
We say an initial tableau is
\begin{enumerate}
    \item of the first kind, if $a-b\equiv1~(3)$,
    \item of the second kind, if $a-b\equiv0~(3)$.
\end{enumerate}
For an initial tableau of the first kind, we let $v_{t} = (0,1,2)'$; otherwise $v_{t} = (0,2,1)'$.
\begin{lemma}\label{lem: prob1 3}
    The seed $S(3) = \{(t,v_{t}) : t~\text{is initial}\}$ is a solution to Problem \ref{prob: prob1}.
\end{lemma}
\begin{proof}
    Every tableau of a chain initiated at a tableau $t = \colTabInline{0,b_0,a_0}$ of the first kind is of the form
    \[
    \colTab{i,b_0+i,a_0+i}
    \quad\text{or}\quad
    \colTab{i,b_0+i,a_0+i+1}
    \quad\text{or}\quad
    \colTab{i,b_0+i+1,a_0+i+1}
    \]
    for some $i\ge0$. Equivalently, a tableau $\colTabInline{c,b,a}$ is in a chain $P(t_0,v_{t_0})$ initiated at a tableau of the first kind if and only if it verifies
    \begin{enumerate}[label=(\roman*)]
        \item $b - c \equiv 1 ~ (2)$ and $a - b \equiv 1 ~ (3)$, or
        \item $b - c \equiv 1 ~ (2)$ and $a - b \equiv 2 ~ (3)$, or
        \item $b - c \equiv 0 ~ (2)$ and $a - b \equiv 1 ~ (3)$.
    \end{enumerate}
    Similarly, a tableau $\colTabInline{c,b,a}$ is in a chain $P(t_0,v_{t_0})$ initiated at a tableau of the second kind if and only if
    it is of the form
        \[
    \colTab{i,b_0+i,a_0+i}
    \quad\text{or}\quad
    \colTab{i,b_0+i+1,a_0+i}
    \quad\text{or}\quad
    \colTab{i,b_0+i+1,a_0+i+1}
    \]
    for some $i\ge0$ and some $a_0, b_0$ such that $a_0-b_0\equiv0~(3)$, if and only if
    it verifies
    \begin{enumerate}[label=(\roman*)]
    \setcounter{enumi}{3}
        \item $b - c \equiv 1 ~ (2)$ and $a - b \equiv 0 ~ (3)$, or
        \item $b - c \equiv 0 ~ (2)$ and $a - b \equiv 2 ~ (3)$, or
        \item $b - c \equiv 0 ~ (2)$ and $a - b \equiv 0 ~ (3)$.
    \end{enumerate}
    The sets defined by (i)--(vi) are disjoint and partition $\B(3)$.
\end{proof}

The analysis of the previous proof and the formula from \S\ref{sec: Fbot} can be explicitly condensed to
\begin{equation}
    \label{eq: Fbot 3}
F^\bot\,.~\colTab{c,b,a} = \begin{cases}
    \colTabInline{c+1,b,a} & \text{if } b\equiv c~(2), \text{ and } a-b \not\equiv2~(3),\\
    \colTabInline{c,b+1,a} & \text{if } b\not\equiv c~(2), \text{ and } a-b \not\equiv1~(3),\\
    \colTabInline{c,b,a+1} & \text{otherwise}.
\end{cases}
\end{equation}
By inspection, we compute
\[
\varphi^\bot_{a}~\colTabInline{c,b,a} = \big(2 - \delta_{b\not\equiv c(2)} - ((a-b) \bmod 3)\big) \bmod 3,
\]
where $x \bmod r$ is the unique number $0\le y <r$ such that $x\equiv y ~ (r)$.
On the other hand, for $t = \colTabInline{c,b,a}\in\B_r(3)$, we have $t^\downarrow = \pletRowTab{b-c-1} \in \B(1)$. Therefore, 
\[
\varphi_{a-c-1}(t^\downarrow) = (a-c-1) - (b-c-1) - 1 = a-b-1.
\]
We can use Lemma \ref{note: formula for A(t)} to obtain
\begin{equation}\label{eq: A 3}
A(t) = \Big(\big(2 - \delta_{b\not\equiv c(2)} - ((a-b) \bmod 3)\big) \bmod 3\Big) - (a-b-1) + 3c.
\end{equation}
Now $F$ can be obtained as in Definition \ref{de: crystal n}. 
See Figure \ref{fig: crystal 3}, and note that e.g.~the path started at $\colTabInline{0,1,4}$ starts in the $(0,1,0)'$ direction until $\colTabInline{0,3,4}$ and then follows the $(1,1,1)'$ direction. The direction changes because $A(\,\colTabInline{0,3,4}) = 0$.

\begin{lemma}
    The seed $S(3)$ is a solution to Problem \ref{prob: prob2}.
\end{lemma}
\begin{proof}
    Suppose $E_\bot\,.\,t = 0$. Then $t$ is an initial tableau; the expression \eqref{eq: A 3} gives $A(t) \le b-a+1 \le 0$.
\end{proof}
\begin{lemma}
    The seed $S(3)$ is a solution to Problem \ref{prob: prob3}.
\end{lemma}
\begin{proof}
Suppose $A(t) < 0 \le A(F.\,t)$. If $A(t) + 1 = A(F.\,t)$ then $A(F.\,t) = 0$. Suppose that $A(t) + 1 < A(F.\,t)$.

Let $t = \colTabInline{c,b,a}$. Since $A(t) < 0$, we have $F.\,t = \colTabInline{c,b+1,a}$. With the inequality $A(F.\,t) - A(t) > 1$ and \eqref{eq: A 3}  we deduce $b\not\equiv c~(2)$ and $a-b\equiv 1~(3)$. But then $A(t) = 0-(a-b-1) +3c$ is a multiple of $3$ and $$A(F.\,t) = 2-(a-b-1)+1+3c = A(t) + 3$$ too. We conclude $A(t) = -3$ and $A(F.\,t) = 0$.

For the second part of the statement, we use that $A(t) = \varphi^\bot_a(t) - (a-b-1) + 3c$. Given $A(F^k.\,t) \ge 0$ we get $F^{k+1}.\,t = F^\bot\,.\,F^k.\,t$.
Now compute the difference $A(F^{k+1}.\,t) - A(F^k.\,t)$. If $F^{k+1}.\,t = \colTabInline{c+1,b,a}$ then the difference is $2$; if $F^{k+1}.\,t = \colTabInline{c,b+1,a}$ then the difference is $0$; and if $F^{k+1}.\,t = \colTabInline{c,b,a+1}$ then the difference is $1$. So the value of $A$ increases weakly always, and increases strictly unless $F^{k+1}.\,t = \colTabInline{c,b+1,a}$. 
It remains to show that $A(F^2.\,t) > 0$.

Suppose $F^\bot$ acts on $F.\,t$ by increasing the $b$-entry. Hence
\[
t = \colTabInline{c,b,a}, ~~ F.\,t = \colTabInline{c,b+1,a}, ~~ F^2.\,t = \colTabInline{c,b+2,a}.
\]
By \eqref{eq: Fbot 3} we get $(b+1)\not\equiv c~(2)$ and $a-(b+1)\not\equiv 1~(3)$. Hence $A(F.\,t) = 0 = 0-(a-b-2)+3c$. But $A(t) = 2 - (a-b-3) + 3c < 0$; these two expressions are in contradiction.
\end{proof}

The previous three lemmas and Theorem \ref{thm: MAIN} amount to a proof of the next result.
\begin{proposition}
    The operator $F$ obtained from \eqref{eq: Fbot 3} and \eqref{eq: A 3} as in Definition \ref{de: crystal n} defines a crystal operator when restricted to each $\B_r(3)$.
\end{proposition}

\begin{figure}[ht]
    \centering
    \includegraphics[scale=1.6]{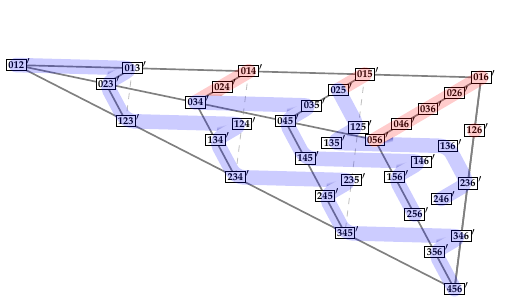}
    \caption{A crystal structure on $\B_6(3) = \SSYT_2(1^3[6])$ as embedded in $\R^3$.}
    \label{fig: crystal 3}
\end{figure}

We can compute the highest weight elements,
\begin{align*}
    \HW(3) &= \left\{t\in\B(3) \ : \ E\,.\,t = 0\right\}\\
    &= \big\{\colTabInline{c,b,a} \ : \ A(t) \le 0, ~~ E^\top\,.\,t = 0 \big\}\\
    &= \left\{\colTabInline{c,b,a} \ : \ 
    A(t) \le 0, {\ }~~~
    b = c+1
    \right\}\\
    &= \left\{\colTabInline{c,b,a} \ : \ 
    \substack{a \ge 4c+2,\\
    a \ne 4c+3,} {\ }~~~
    b = c+1
    \right\}.
\end{align*}

\begin{corollary}
    The plethystic coefficient $a_{1^3[r]}^k$ satisfies
    \[
    a_{1^3[r]}^k = \#\left\{\colTabInline{c,b,a} \in \SSYT_2(1^3[r]) \ : \ 
    \substack{a \ge 4c+2,\\
    a \ne 4c+3,} {\ }~~~
    b = c+1, {\ }~~~
    \substack{3a+2b+c = k, \\ a\le r}
    \right\}.
    \]
\end{corollary}
Before concluding the section for $n\!=\!3$ with the proof of Theorem \ref{thm: char 3}, we introduce an alternative way of illustrating the crystals that give useful geometric intuition. To do so, we slice $\B(3)$ into slices of the form $\B_r(3) \setminus \B_{r-1}(3)$. Each of these slices is embedded naturally in $\R^2$ via $\colTabInline{c,b,r} \mapsto (c,b)'$. We take the convention that a path in a slice always travels South. When a path `jumps' from one slice to another, we do not draw anything: if the slices are lined up then the jump is clear. Compare Figures \ref{fig: crystal 3} and \ref{fig: crystal 3 slices}. For instance, the path starting at $\colTabInline{0,1,2}$ jumps to $\colTabInline{0,1,3}$, travels South till $\colTabInline{1,2,3}$, jumps to $\colTabInline{1,2,4}$, etc.

\begin{figure}[ht]
    \centering
    \includegraphics[valign=t, scale=1.35]{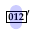}
    \includegraphics[valign=t, scale=1.35]{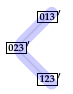}
    \includegraphics[valign=t, scale=1.35]{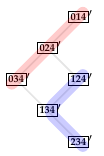}
    \includegraphics[valign=t, scale=1.35]{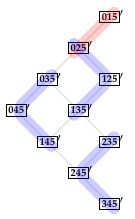}
    \includegraphics[valign=t, scale=1.35]{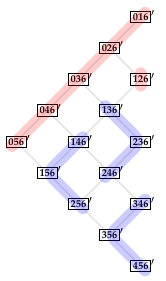}    
    \caption{A crystal structure on $\B_6(3) = \SSYT_2(1^3[6])$ as embedded in $\R^2$ slice by slice.}
    \label{fig: crystal 3 slices}
\end{figure}

\begin{proof}[Proof of Theorem \ref{thm: char 3}.]
The crystal of $\B_r(3)$ is obtained from that of $\B_{r-1}(3)$ by extending each path by 3 nodes, and then adding some other components that are governed by $F^\top$ (which is given by the operator $F$ on $\B(1)$). (For instance, the crystal on $\B_6(3)$ in Figures \ref{fig: crystal 3} and  \ref{fig: crystal 3 slices} is obtained from the crystal on $\B_5(3)$ by extending the paths that terminate in $\colTabInline{1,3,5}$, $\colTabInline{1,4,5}$ and $\colTabInline{3,4,5}$ and then adding the chains initiated at $\colTabInline{0,1,6}$ and $\colTabInline{1,2,6}$.)

Consider a tableau $t = \colTabInline{c,b,r}$ in the $r$th slice and such that $A(t) = 0$. By reducing \eqref{eq: A 3} modulo 3, we deduce that $t$ is in either the set (i) or (vi) from the proof of Lemma \ref{lem: prob1 3}. 

If $t$ is in (i), then $A(t) = 0 - (r-b-1) + 3c$ and hence $b = r-1-3c$. 
By definition of~(i) we have $(r-1-3c) - c \equiv 1 ~(2)$ and we deduce $r$ is even.
Conversely, for all $r$ even, each tableau of the form 
\[
t = \colTabInline{k, r-1-3k, r} \text{~for some~} k\ge0 \text{~such that~} 4k < r-1
\]
satisfies $A(t) = 0$. There is exactly one tableau in each chain for which $A(-)$ vanishes. Given $t=\colTabInline{k, r-1-3k, r}$ we compute $\varepsilon(t) = \varepsilon(t^\downarrow) = b-c-1 = r-2-4k$ and $\varphi_r(t) = 0$. Thus the chain $t$ belongs to has $r-1-4k$ nodes: it contributes to the character with a summand $[r-1-4k]$.

If $t$ is in (vi), then $A(t) = 2 - (r-b-1) + 3c$ and hence $b = r-3c-3$. Going back to $t\in\text{(vi)}$, we deduce $r$ odd. Note that $\varepsilon(t) = r-4c-4$ and $\varphi_r(t) = 2$. Conversely, for all $r$ odd, each tableau of the form
\[
t = \colTabInline{k, r-3-3k, r} \text{~for some~} k\ge0 \text{~such that~} 4k < r-5
\]
satisfies $A(t) = 0$. Each of these is in a chain which contributes to the character with a summand $[r-3-4k]_{+2} = [r-1-4k]$.

Once again, since every chain contains exactly one tableau such that $A(-)$ vanishes, we obtain the desired formula.
\end{proof} 

We remark how this can be seen from Figures \ref{fig: crystal 3} and \ref{fig: crystal 3 slices}. The chains initiated at $\colTabInline{0,1,6}$ and $\colTabInline{1,2,6}$ are in the $(0,1,0)'$ direction, and together contribute to the character of $\Alt^3\Sym^6\CC^2$ with $[5] + [1]$. Truncating the figure to its four first slices, we see that the chain initiated at $\colTabInline{0,1,5}$ has one tableau $\colTabInline{0,2,5}$ for which $A(-)$ vanishes and contributes to the character of $\Alt^3\Sym^5\CC^2$ with $[4]$.



\subsection{Crystals of \texorpdfstring{$\Alt^4\Sym^r\CC^2$}{Alt\^{}4 Sym\^{}r C\^{}2}}
Let a tableau $t = \colTabInline{d,c,b,a}$ be \emph{initial} if
$d=0$, and $a$ and $c$ are odd. Moreover, we say it is
\begin{itemize}
    \item initial of the first kind if $b$ is even, and
    \item initial of the second kind if $b$ is odd.
\end{itemize}
We let $v_{t} = (0,1,2,3)'$ and $v_{t} = (0,3,2,1)'$ be the offsets of the first and second kind, respectively.
\begin{lemma}\label{lem: prob1 4}
    The seed $S(4) = \{(t,v_{t}) \ : \ t ~\text{is initial}\}$ is a solution to Problems \ref{prob: prob1}.
\end{lemma}
\begin{proof}
    Every vector of a path $P(t_0, v_{t_0})$ initiated at an initial tableau $\colTabInline{0,c_0,b_0,a_0}$ of the first kind is of the form
    \[
    \colTab{i,c_0+i,b_0+i,a_0+i}
    \quad\text{or}\quad
    \colTab{i,c_0+i,b_0+i,a_0+i+1}
    \quad\text{or}\quad
    \colTab{i,c_0+i,b_0+i+1,a_0+i+1}
    \quad\text{or}\quad
    \colTab{i,c_0+i+1,b_0+i+1,a_0+i+1}
    \]
    for some $i\ge0$.
    Therefore, a tableau $\colTabInline{d,c,b,a}$ is in a chain $P(t_0,v_{t_0})$ initiated at a tableau of the first kind if and only if it satisfies
    \begin{enumerate}[label=(\roman*)]
        \item $c-d\equiv1~(2)$, $b-c\equiv1~(2)$ and $a-b\equiv1~(2)$, or
        \item $c-d\equiv1~(2)$, $b-c\equiv1~(2)$ and $a-b\equiv0~(2)$, or
        \item $c-d\equiv1~(2)$, $b-c\equiv0~(2)$ and $a-b\equiv1~(2)$, or
        \item $c-d\equiv0~(2)$, $b-c\equiv1~(2)$ and $a-b\equiv1~(2)$.
    \end{enumerate}
    Similarly, a vector $\colTabInline{d,c,b,a}$ is in a chain $P(t_0,v_{t_0})$ initiated at a tableau of the second kind if and only if 
    it is of the form
        \[
    \colTab{i,c_0+i,b_0+i,a_0+i}
    \quad\text{or}\quad
    \colTab{i,c_0+i+1,b_0+i,a_0+i}
    \quad\text{or}\quad
    \colTab{i,c_0+i+1,b_0+i+1,a_0+i}
    \quad\text{or}\quad
    \colTab{i,c_0+i+1,b_0+i+1,a_0+i+1}
    \]
    for some $i\ge0$ and some $a_0, b_0, c_0$, all odd;
    if and only if
    it satisfies
    \begin{enumerate}[label=(\roman*)]
    \setcounter{enumi}{4}
        \item $c-d\equiv1~(2)$, $b-c\equiv0~(2)$ and $a-b\equiv0~(2)$, or
        \item $c-d\equiv0~(2)$, $b-c\equiv1~(2)$ and $a-b\equiv0~(2)$, or
        \item $c-d\equiv0~(2)$, $b-c\equiv0~(2)$ and $a-b\equiv1~(2)$, or
        \item $c-d\equiv0~(2)$, $b-c\equiv0~(2)$ and $a-b\equiv0~(2)$.
    \end{enumerate}
    The sets defined by (i)--(viii) are disjoint and partition $\B(4)$.
\end{proof}
Explicitly, we get a formula
\begin{equation}\label{eq: Fbot 4}
F^\bot\,.~\colTab{d,c,b,a} = \begin{cases}
    \colTabInline{d+1,c,b,a} & \text{if } a\equiv c\equiv d\, (2),
    \\
    \colTabInline{d,c+1,b,a} & \text{if } b\equiv c\not\equiv d\, (2),
    \\
    \colTabInline{d,c,b+1,a} & \text{if } a\equiv b\not\equiv c\, (2),
    \\
    \colTabInline{d,c,b,a+1} & \text{if } a\not\equiv b\equiv d\, (2).
        \end{cases}
\end{equation}
By inspection, we compute
\[
\varphi^\bot_{a}\,\colTabInline{d,c,b,a} = 
(\delta_{a\not\equiv b(2)} + 2\delta_{b\not\equiv c(2)} - \delta_{c\not\equiv d(2)} - 1 - \delta_{t \in \text{(i)--(iv)}}) \bmod 4,
\]
where $\delta_{t\in\text{(i)--(iv)}}$ checks whether the tableau $t = \colTabInline{d,c,b,a}$ is in the parts (i)--(iv) from the previous proof. Note that these are the tableaux in which there are at least two changes of parity when read in order.
On the other hand, we have $t^\downarrow = \colTabInline{c-d-1,b-d-1}\in\B(2)$ and therefore
\[
\varphi_{a-d-1}(t^\downarrow)
= 2(a-b-1) + \delta_{b\equiv c(2)}.
\]
We can now use the formula of Lemma \ref{note: formula for A(t)} to obtain
\begin{equation}\label{eq: A 4}
A(t) =
\varphi^\bot_{a}(t)
- 2(a - b - 1) - \delta_{b\equiv c(2)} + 4d.
\end{equation}
The operator $F$ can be computed as in Definition \ref{de: crystal n}.
\begin{lemma}
    The seed $S(4)$ is a solution to Problem \ref{prob: prob2}.
\end{lemma}
\begin{proof}
    Suppose $E_\bot\,.\,t = 0$. Then $t$ is an initial tableau. If $t$ is of the first kind, then $t\in\text{(i)}$ and $A(t) = 0 - 2(a-b-1) - 0 - 0 \le 0$. If $t$ is of the second kind, then $t \in \text{(v)}$; we get $A(t) = 2 - 2(a-b-1) - 1 - 0$ and also that $a$ and $b$ are of the same parity, hence $A(t) \le 0$.
\end{proof}
\begin{lemma}
    The seed $S(4)$ is a solution to Problem \ref{prob: prob3}.
\end{lemma}
\begin{proof}
Suppose $A(t) < 0 \le A(F.\,t)$. If $A(t) + 1 = A(F.\,t)$ then $A(F.\,t) = 0$. Suppose that $A(t) + 1 < A(F.\,t)$.

Let $t = \colTabInline{d,c,b,a}$. With the inequality $A(F.\,t) - A(t) > 1$ and \eqref{eq: A 4}  we deduce 
\[
a\not\equiv b \equiv c ~ (2) \quad\text{and}\quad F.\,t = \colTabInline{d,c,b+1,a}.
\]
Regardless of whether $c\equiv d~(2)$ or not, we can compute $A(F.\,t) - A(t) = 2$ but $A(t) \equiv 1 ~(4)$, which is a contradiction to $A(t) < 0 \le A(F.\,t)$.

For the second part of the statement, we note that $A(F^\bot.\,t) - A(t)$ is $1$ if $F^\bot$ acts by increasing the $a$-entry, $3$ if it acts by increasing the $d$-entry, and $0$ otherwise. So the value of $A(F^k.\,t)$ increases weakly for $k\ge1$, and increases strictly unless the operator $F$ acts by increasing the $b$- or the $c$-entry of $F^{k-1}.\,t$.
It remains to show that $A(F^2.\,t) > 0$.

If $F^2.\,t$ is obtained from $F.\,t$ by increasing either the $a$- or $d$-entry we are done; we suppose the contrary and look for a contradiction.
Note that if $F.\,t = \colTabInline{d,c,b+1,a}$ then $b\not\equiv c~(2)$ and hence $F^2.\,t$ cannot possibly be obtained from increasing the $b$-entry by looking at \eqref{eq: Fbot 4}. Similarly, $F^2.\,t$ cannot be $\colTabInline{d,c+2,b,a}$. We deduce that $F^2.\,t = \colTabInline{d,c+1,b+1,a}$. 

If $F.\,t = \colTabInline{d,c,b+1,a}$ then we conclude $b\not\equiv c \not\equiv d ~ (2)$. That is, $t$ is in parts (i) or (iv). A computation yields $A(F.\,t) - A(t) = 1$. Since $A(F.\,t) = 0$, we must have $A(t)\equiv3~(4)$. But \eqref{eq: A 4} gives $A(t) \equiv \delta_{a\not\equiv b~(2)}~(4)$, which is the contradiction we are after.

If $F.\,t = \colTabInline{d,c+1,b,a}$ then we conclude $a\equiv b \equiv c ~ (2)$. That is, $t$ is in parts (v) or (viii). A computation yields $A(F.\,t) - A(t) = 2 + \delta_{c\not\equiv d~(2)}$. We compute $A(t) \bmod 4$ to be $3$ in part (v) and $0$ in part (viii), giving contradictions in the same manner as before.
\end{proof}

The previous three lemmas and Theorem \ref{thm: MAIN} amount to a proof of the next result.
\begin{proposition}
    The operator $F$ obtained from \eqref{eq: Fbot 4} and \eqref{eq: A 4} as in Definition \ref{de: crystal n} defines a crystal operator when restricted to each $\B_r(4)$.
\end{proposition}

We can compute the highest weight elements,
\begin{align*}
    \HW(4) &= \left\{t \in \B(4) \ : \ E\,.\,t=0\right\}\\
    &= \left\{t \in \B(4) \ : \ A(t) \le 0, ~~ E^\top\,.\,t=0\right\}\\
    &= \Big\{\colTabInline{d,c,b,a} \ : \ A(t) \le 0, ~~ \substack{c = d+1,\\ b\not\equiv c ~(2)}\Big\}\\
    &= \Big\{\colTabInline{d,c,b,a} \ : \ \substack{a \ge b+2d+1,\\
    a \ne b+2d+2,} {\ }~ \substack{c = d+1,\\ b\not\equiv c ~(2)}\Big\}.
\end{align*}
\begin{corollary}
    The plethystic coefficient $a_{1^4[r]}^k$ satisfies
    \[
    a_{1^4[r]}^k = \#\Big\{\colTabInline{d,c,b,a}\in\HW(4) \ : \ 
    \substack{a \ge b+2d+1,\\
    a \ne b+2d+2,} {\ }~ \substack{c = d+1,\\ b\not\equiv c ~(2),} {\ }~ 
    \substack{4a+3b+2c+d = k,\\ a\le r}\Big\}.
    \]
\end{corollary}

We can visualize the crystals in a similar manner as before, by embedding $\B(4)$ into $\R^4$, then breaking it into tetrahedrons $\B_r(4) \setminus \B_{r-1}(4)$, and slicing the tetrahedrons into triangles. 
Hence each slice is the set of tableaux of the form $\colTabInline{d_0, c, b, a_0}$ for some fixed $a_0$ and $d_0$.
We take the convention that a path within a slice always travels South. We do this in Figure \ref{fig: crystal 4}, where a path ends in a circle if it jumps to the next slice to the right, and a path ending on a square jumps to the next slice down. Since the slices are aligned, the jumps should be clear. For instance, the path started at $\colTabInline{0,1,2,3}$ jumps down to $\colTabInline{0,1,2,4}$, travels South till $\colTabInline{0,2,3,4}$, jumps right to $\colTabInline{1,2,3,4}$, jumps down to $\colTabInline{1,2,3,5}$, etc.

\begin{figure}[p]
    \centering
    \includegraphics[page = 6, scale = .9]{Images/B8_4_large.pdf}
    \includegraphics[page = 5, scale = .9]{Images/B8_4_large.pdf}
    \includegraphics[page = 4, scale = .9]{Images/B8_4_large.pdf}
    \includegraphics[page = 3, scale = .9]{Images/B8_4_large.pdf}
    \includegraphics[page = 2, scale = .9]{Images/B8_4_large.pdf}
    \includegraphics[page = 1, scale = .9]{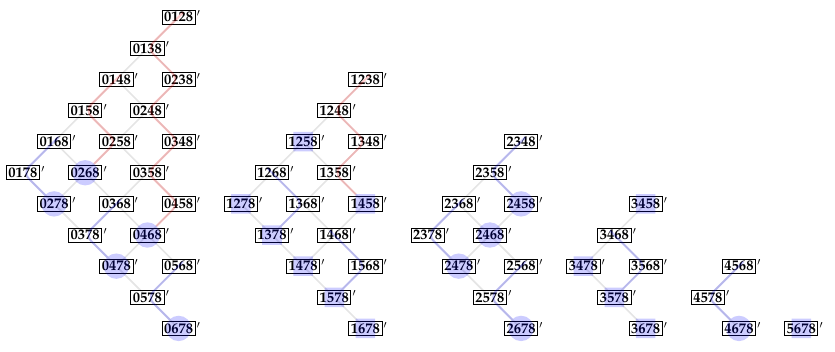}
    \caption{A crystal structure on $\B_8(4)$. Each path is directed south, a path ending on a circle jumps to the right, and a path ending on a square jumps down.}
    \label{fig: crystal 4}
\end{figure}

The figure contains geometric intuition that helps when showing Theorem \ref{thm: char 4} on the character of $\Alt^4\Sym^r\CC^2$.

\begin{proof}[Proof of Theorem \ref{thm: char 4}]
Some components of $\B_{r}(4)$ are obtained by extending those of $\B_{r-1}(4)$. (In Figure \ref{fig: crystal 4}, the crystal for $\B_{r-1}(4)$ is obtained by truncating the picture to the first $r-1$ rows; each path ending on a square gets extended by $4$ new nodes in the next row of slices.)
The remaining components are governed by $\smash{F^\top}$, which is given by the crystal operator on $\B(2)$.

For the rest of the proof we suppose $r$ is odd and show
\[
\qbinom{r+1}{4} = \qbinom{r}{4}_{+4} + \sum_{k\ge0}\qbinom{r-6k-1}{2} + \sum_{k\ge0}\qbinom{r-6k-7}{2}_{+6}.
\]
The proof for even $r$ is similar.

Let $t = \colTabInline{d,c,b,r}$ be a tableau in the $r$th row of slices and suppose $A(t) = 0$. Since $r$ is odd, $t$ is seen to be in sets (i) or (viii) from Lemma \ref{lem: prob1 4} by explicit computation.

If $t$ is in (i) then $\varphi_{r}(t) = \varphi^\bot_{r}(t) = 0$ and the formula from \eqref{eq: A 4} gives $4d = \varphi_{r-d-1}(t^\downarrow) = 2(r-b-1)$.
Reordering terms, $b-d-1 = r-3d-2$. This means that $u\mapsto u^\downarrow$ maps the chain ending in $t$ to a chain of $\B_{r-3d-2}(2)$. Since $r+1$ is odd and $t \in \text{(i)}$, then $d = 2k$ is even and we can write $\B_{r-6k-2}(2)$. Therefore all tableaux $t$ such that $A(t) = 0$ and $t\in\text{(i)}$ introduce a summand $\qbinom{(r-6k-2)+1}{2}$ to the character.

Now suppose $t\in\text{(viii)}$. 
Since $r$ is odd then $d$ is odd and $\varphi_{r}(t) = 3$ by explicit computation. Using \eqref{eq: A 4} we obtain $3 + 4d = 2(r-b-1) + 1$. 
But we can also compute $\varepsilon(t) = \varepsilon(t^\downarrow) = 2(b-c)-1$. Since $b\equiv c ~(2)$, we have $\varepsilon(t)\ge3$, and thus we know $E^3.\,t$ exists. We prefer thinking of the chain containing $t$ as passing by $E^3.\,t = \colTabInline{d,c-1,b-2,r}$ and getting extended by $6$ nodes. Reordering the expression obtained using \eqref{eq: A 4} for $t$, we obtain $(b-2)-d-1 = r - 3d - 5$. Write $d = 2k+1$; the chain passing passing by $E^3.\,t$ is obtained from $\B_{r-6k-8}(2)$ and then extended by $6$. Hence all tableaux $t$ such that $A(t) = 0$ and $t\in\text{(viii)}$ introduce a summand $\qbinom{(r-6k-8)+1}{2}_{+6}$ to the character.

Since every chain contains exactly one tableau such that $A(t) = 0$, and we have accounted for all such tableaux in $\B_r(4)$, we are done.
%
%
%
%
\end{proof}

We remark how this character formula can be read from Figure \ref{fig: crystal 4}. For instance, some tableaux $t$ such that $A(t) = 0$ and $t\in\text{(i)}$ are $\colTabInline{0,1,6,7}$, $\colTabInline{0,3,6,7}$, $\colTabInline{0,5,6,7}$. Compare the slice $\colTabInline{0,b,c,7}$ of the figure with the crystal on $\B_5(2)$ from Figure \ref{subfig: B5(2)}.

\section{Final remarks}\label{sec: final}


The number of constituents of $\Alt^n\Sym^r\CC^2$ is the number of tableaux in $\SSYT_2(1^n[r])$ of weight $0$ if $n+r$ is odd, or weight $-1$ if $n+r$ is even. Through the bijection $\Psi$ of \eqref{eq: bijection B(n) L(n)} this number is equal to $\#\{\lambda\in L(n,m) \ : \ \lambda\vdash\lfloor{nm/2}\rfloor\}$, where $r+1=n+m$. Indeed this is the number of partitions in the centre rank of $L(n,m)$, and in a symmetric chain decomposition of $L(n,m)$ each chain visits the centre rank exactly once.

\begin{corollary}\label{cor: number 2}
     The cardinality of $\{\lambda\in L(n,m) \ : \ \lambda\vdash\lfloor{nm/2}\rfloor\}$ is
    \begin{enumerate}
        \item $\lfloor{(r+1)/2}\rfloor$ for $n\!=\!2$, and
        \item $\lfloor{(r+1)^2/8}\rfloor$ for $n\!=\!3$.
    \end{enumerate}       
\end{corollary}
\begin{proof}[First proof of Corollary \ref{cor: number 2}]
    We proceed by double counting.
    We count the number of components in our crystals by counting the partitions in the centre rank of $L(n,m)$. But we can similarly count the number of highest weight tableaux $\HW(n)$ intersecting $\B_r(n)$, since again there is one such tableau per chain in a symmetric chain decomposition of $L(n,m)$. Given the expression of $\HW(2)$ from the previous section, the statement for $n\!=\!2$ is now clear. The set $\HW(3)\cap\B_r(3)$ is, up to the shift
    \[
    \colTabInline{4c+2,c+1,c} \mapsto \colTabInline{4c+1,c+1,c},
    \]
    the set of integer points in a right triangle of a plane of $\R^3$ with base $r+1$ and height $(r+1)/4$.
\end{proof}
The set $\HW(4)\cap\B_r(4)$ is the integer points of a tetrahedron of volume ${(r+1)^3/18}$ \textit{to which roughly half the points were removed} by imposing two coordinates to be of different parity. Thus the number of constituents is roughly $\lfloor{(r+1)^3/36}\rfloor$ for $n\!=\!4$. The precise number is $\lfloor(2 r^{3} - 3 r^{2} + 6 r + 27)/72\rfloor$ \cite[p.~69]{AF}.

We have yet a third way of computing the number of constituents of $\Alt^n\Sym^r\CC^2$, using the character formulas obtained for $n=2$ and $n=3$ (this latter is Theorem \ref{thm: char 3}). Remark that $[k]|_{q=1} = k$ and $\qbinom{k}{n}|_{q=1} = \binom{k}{n}$.

\begin{proof}[Second proof of Corollary \ref{cor: number 2}]
If $f$ is a character of a representation $V$ (hence $f$ is a Laurent polynomial on $q^{1/2}$), then the specialisation $f|_{q=1}$ is the dimension of $V$, and $(f_{+j} - f)|_{q=1}$ is $j$ times the number of constituents of $f$.
The number of constituents of $\Alt^n\Sym^r\CC^2$ is therefore
\[\frac{1}{n}\left.\left(\qbinom{r+1}{n}_{+n} - \qbinom{r+1}{n}\right)\right|_{q=1}.\]
For $n = 2$ we obtain
the formula $\smash{\qbinom{r+2}{2} = \qbinom{r+1}{2}_{+2} + \delta_{r~\text{even}}\cdot[1]}$ from \eqref{eq: recursion n=2} and hence the number of constituents of $\Alt^2\Sym^r\CC^2$ is 
\[
\frac{1}{2}\Bigg(\binom{r+2}{2} - \delta_{r~\text{even}} - \binom{r+1}{2}\Bigg) = \frac{r + 1 - \delta_{r~\text{even}}}{2} = \left\lfloor\frac{r+1}{2}\right\rfloor.
\]
For $n = 3$ Theorem \ref{thm: char 3}, Pascal's identity, and the arithmetic series formula give
\begin{multline*}
    \frac{1}{3}\Bigg(\binom{r+2}{3} - \!\!\!\!\!\!\sum_{\substack{k\ge0\\ 4k<r-2\delta_{r~\text{even}}}} \!\!\!\!\!\!(r-4k) - \binom{r+1}{3}\Bigg) =
    \frac{1}{3}\Bigg(\binom{r+1}{2} - \!\!\!\!\!\!\sum_{\substack{k\ge0\\ 4k<r-2\delta_{r~\text{even}}}} \!\!\!\!\!\!(r-4k)\Bigg)\\ 
    =
    \frac{r(r+1)}{6} - \frac{1}{3}\left\lceil\frac{r-2\delta_{r~\text{even}}}{4}\right\rceil\cdot\Big(r + 2 - 2\left\lceil\frac{r-2\delta_{r~\text{even}}}{4}\right\rceil\Big).
\end{multline*}
Distinguishing by cases ($r$ even or odd) and straightforward algebraic manipulations simplify this expression to $\lfloor{(r+1)^2/8}\rfloor$, as desired.
\end{proof}

For $n = 4$ the computations are cumbersome, and so we prefer to omit them.
\medskip

We were not able to prove or disprove the applicability of our crystal framework for $n\!=\!5$ and beyond. The number of choices needed to find seeds grows quickly with $n$. Standardising these choices might be the only thing keeping us away from a solution to Problems \ref{prob: plethysm} and \ref{prob: SCD}.

\begin{appendices}
    
\section{Recovering related constructions}\label{sec: recover}

Transporting our operators $F$ through the bijection $\Psi$ of \eqref{eq: bijection B(n) L(n)}, we find symmetric chain decompositions of $L(n,r+1-n) = L(n,m)$ for $n \le 4$. In the next two propositions, we show that the crystals given in the previous section are a new description of the symmetric chain decompositions of $L(n,m)$ given in \cite{OSSZ}.
We remark once again that this is unexpected, since (i) both approaches do not coincide, and (ii) the `desirable properties' imposed in \cite{OSSZ} are not all a priori required in our constructions. Moreover, we remind the reader that our decompositions for $n\!=\!3$ and $4$ are examples of one unique construction, which is not the case in \cite{OSSZ}. All these facts might be pointing to some uniqueness result, which we leave for future exploration.
\begin{proposition}\label{p: recover 3}
    For all $t\in\B(3)$, we have $\Psi(F.t) = f.\Psi(t)$, where $f$ is defined in \cite[Thm.~24]{OSSZ}.
\end{proposition}
\begin{proof}
    We begin by comparing the sets of highest weight tableaux. Our description gives
    \begin{align*}
        \HW(3) = \left\{\colTabInline{c,b,a} \ : \ 
    \substack{a \ge 4c+2,\\
    a \ne 4c+3,} {\ }~~~
    b = c+1
    \right\}.
    \end{align*}
    Apply the bijection $\Psi:\B_r(3) \to L(3, r+1-3)$ to obtain
    \[
    \left\{(3^z2^y1^x) \ :\
    \substack{x \ge 3z,\\
    x \ne 3z+1,} {\ }~~~
    y = 0
    \right\}.
    \]
    These are the minimal elements described in \cite[Thm.~24(2)]{OSSZ}. (Caveat lector: our bijection changes highest weight for lowest weight.) Next \cite{OSSZ} proceeds by induction, considering the poset isomorphism given by
    \begin{align*}
        L(3,h-4) &\to L(3,h)\\
        \lambda &\mapsto \lambda+(3,1^3).
    \end{align*}
    We claim that our crystal structure is also compatible with this recursion. 
    Translating this map through $\Psi^{-1}$, it becomes $\colTabInline{c,b,a}\mapsto \colTabInline{c+1,b+1,a+4}$, which we write as $t \mapsto t + (1,1,4)'$.
    
    It is apparent from inspection of \eqref{eq: Fbot 3} that $F^\bot(t + (1,1,4)') = F^\bot(t) + (1,1,4)'$.
    Furthermore, we have the following formulas
    \begin{align*}
        \varphi^\bot_{r}(t+(1,1,4)') &= \varphi^\bot_{r}(t) - 12\\
        \varepsilon^\top(t+(1,1,4)') &= 
        \varepsilon^\top(t)\\
        \wt_{r}(t+(1,1,4)') &= \wt_{r}(t) - 12.
    \end{align*}
    Consequently, $A(t) = A(t+(1,1,4)')$. Altogether,
    \[
    F(t+(1,1,4)') = F(t) + (1,1,4)',
    \]
    as desired.
    To finish, \cite[Prop.~27]{OSSZ} now describes a symmetric chain decomposition of the complement $L'(3,h)$ of the image of the above isomorphism. That is, the poset whose underlying set fits in
    \[
    L(3,h) = L'(3,h) \sqcup
    \{\lambda+(3,1^3)\ : \ \lambda\in L(3,h-4)\}.
    \]
    To do this, formulas are given for the crystal operator, which they denote by $f$. The map is defined as a map by parts on 8 parts, denoted (1)--(8). The first two parts correspond to tableaux on $\Top(3)$, and the remaining 6 parts correspond to the sets (i)--(vi) of $\Bottom(3)$ defined in the proof of Lemma \ref{lem: prob1 3}.
    For instance, the set (3) is given by the partitions $\lambda = (3^z2^y1^x)$ with $z = 0$ and $y$ even. Noting that $0\le z < 3$ in $L'(3,h)$, and translating to our setting through $\Psi^{-1}$, we get the set of tableaux $t = \colTabInline{c,b,a}$ given by 
    \[
    b-c\equiv1~(2)
    \quad\text{and}\quad
    a-b\equiv1~(3).
    \]
    This is the set (i) of Lemma \ref{lem: prob1 3}. They let $f.(3^z2^y1^x) = (3^z2^y1^{x+1})$ for partitions $\lambda = (3^z2^y1^x)$ in (3), and we let $F.t = t + (0,0,1)'$ for $t$ in (i). Through $\Psi^{-1}$ these are the same operator,
    \[
    f.\Psi(t) = \Psi(F.\,t)
    \quad\text{for }t\text{ in (i).}
    \]
    The full correspondence between the sets (3)--(8) of \cite[Prop.~27]{OSSZ} and the sets (i)--(vi) of Lemma \ref{lem: prob1 3} is as follows
    \[
    (3) \leftrightarrow \textnormal{(i)},\quad
    (4) \leftrightarrow \textnormal{(iii)},\quad
    (5) \leftrightarrow \textnormal{(ii)},\quad
    (6) \leftrightarrow \textnormal{(v)},\quad
    (7) \leftrightarrow \textnormal{(iv)},\quad
    (8) \leftrightarrow \textnormal{(vi)}.
    \]
    It remains to check $f.\Psi(t) = \Psi(F.t)$ in the five latter cases, which we leave to the reader.
\end{proof}

\begin{proposition}\label{p: recover 4}
    For all $t\in\B(4)$, we have $\Psi(F.t) = f.\Psi(t)$, where $f$ is defined in \cite[Thm.~31]{OSSZ}.
\end{proposition}
\begin{proof}
    We begin by comparing the sets of highest weight tableaux. Our description gives
    \begin{align*}
        \HW(4) &= \Big\{\colTabInline{d,c,b,a} \ : \ \substack{a \ge b+2d+1,\\
        a \ne b+2d+2,} {\ }~ \substack{c = d+1,\\ b\not\equiv c ~(2)}\Big\}.
    \end{align*}
    Apply the bijection $\Psi:\B_r(4) \to L(4, r+1-4)$ to obtain
    \[
        \Big\{(4^w3^z2^y1^x) \ :\ \substack{x \ge 2w,\\
        x \ne 2w+1,} {\ }~ \substack{z = 0,\\ y\equiv 0 ~(2)}\Big\}.
    \]
    These are the minimal elements described in \cite[Thm.~31(2)]{OSSZ}.
    Next \cite{OSSZ} proceeds by double induction, considering the poset isomorphisms given by
    \begin{align*}
        L(4,h-3) &\to L(4,h)
        &L'(4,h-2) &\to L'(4,h)\\
        \lambda &\mapsto \lambda+(4,1^2)
        &\lambda &\mapsto \lambda+(2^2),
    \end{align*}
    where $L'(4,h)$ is the complement of the image of the first isomorphism; let $L''(4,h)$ be the complement of the image of the second one. That is,
    \begin{align*}
        L(4,h) &= L'(4,h) \sqcup
    \{\lambda+(4,1^2)\ : \ \lambda\in L(4,h-3)\} \quad\text{and}\\
        L'(4,h) &= L''(4,h) \sqcup
    \{\lambda+(2^2)\ : \ \lambda\in L(4,h-2)\}.
    \end{align*}
    We claim that our crystal structure is also compatible with these recursions.
    Translating the maps through $\Psi^{-1}$, they become 
    \[
    t\mapsto t+(1,1,1,3)'
    \quad\text{and}\quad
    t\mapsto t+(0,0,2,2)'.
    \]
    It is apparent from inspection of our expression of \eqref{eq: Fbot 4} that $F^\bot(t + (1,1,1,3)') = F^\bot(t) + (1,1,1,3)'$ and similarly for $(0,0,2,2)'$. Furthermore, we have the following formulas
    \begin{align*}
        \varphi^\bot_{r}(t+(1,1,1,3)') &= \varphi^\bot_{r}(t) - 12\\
        \varepsilon^\top(t+(1,1,1,3)') &= 
        \varepsilon^\top(t)\\
        \wt_{r}(t+(1,1,1,3)') &= \wt_{r}(t) - 12.
    \end{align*}
    In a similar way, one can analyse the map $t \mapsto t+(0,0,2,2)'$ and obtain analogous formulas, where the coefficient $12$ now becomes $8$.
    Consequently, $A(t) = A(t+(1,1,1,3)')$ and $A(t) = A(t+(0,0,2,2)')$. Altogether,
    \[
    F(t+(1,1,1,3)') = F(t) + (1,1,1,3)'
    \quad\text{and}\quad
    F(t+(0,0,2,2)') = F(t) + (0,0,2,2)'.
    \]
To finish, \cite[Prop.~33]{OSSZ} now describes a symmetric chain decomposition of $L''(4,h)$. 
    To do this, formulas are given for the crystal operator $f$. The map is defined as a map by parts on 10 parts, denoted (1)--(10). The first two parts correspond to tableaux on $\Top(4)$, and the remaining 8 parts correspond to the sets (i)--(viii) of $\Bottom(4)$ defined in the proof of Lemma \ref{lem: prob1 4}. (Note that any partition $(4^w3^z2^y1^x)$ in $L''(4,h)$ satisfies $0\le x < 2$ and $0\le y < 2$.)
    The full correspondence between the sets (3)--(10) of \cite[Prop.~33]{OSSZ} and the sets (i)--(viii) of Lemma \ref{lem: prob1 4} is as follows
    \begin{gather*}
    (3) \leftrightarrow \textnormal{(i)},\quad
    (4) \leftrightarrow \textnormal{(ii)},\quad
    (5) \leftrightarrow \textnormal{(iii)},\quad
    (6) \leftrightarrow \textnormal{(iv)},\\
    (7) \leftrightarrow \textnormal{(viii)},\quad
    (8) \leftrightarrow \textnormal{(v)},\quad
    (9) \leftrightarrow \textnormal{(vi)},\quad
    (10) \leftrightarrow \textnormal{(vii)}.
    \end{gather*}
    It remains to check that for each of these sets (1)--(10) we have $\Psi(F.\,t) = f.\Psi(t)$.
    For instance, the set (4) is given by the partitions $\lambda = (4^w3^z2^y1^x)$ with $x = 1$, $y = 0$, and $z$ even. They let $f.(4^w3^z2^y1^x) = (4^w3^z2^{y+1}1^{x-1})$ for partitions in (4), and we let $F.t = t + (0,1,0,0)'$ for $t$ in (ii). Through $\Psi$ these are the same operator. We leave the remaining cases to the reader.
\end{proof}

\end{appendices}

\section*{Acknowledgements}
We thank Mark Wildon for proposing the problem that started this research as well as his guidance throughout; Rosa Orellana, Franco Saliola, Anne Schilling, and Mike Zabrocki for many discussions and suggestions; Michał Szwej for his proofreading. We thank Greta Panova for making us aware of her and Igor Pak's combinatorial interpretation of $\sl_2$ plethystic coefficients in September 2025.

\bibliographystyle{halpha}
\bibliography{Bibliography}

\end{document}